\documentclass[hidelinks,onefignum,onetabnum]{siamart250211}
\usepackage{lipsum}
\usepackage{amsfonts}
\usepackage{graphicx}
\usepackage{epstopdf}
\usepackage{algorithmic}
\usepackage{placeins}
\usepackage{url}
\ifpdf
  \DeclareGraphicsExtensions{.eps,.pdf,.png,.jpg}
\else
  \DeclareGraphicsExtensions{.eps}
\fi

\newsiamremark{remark}{Remark}
\newsiamremark{hypothesis}{Hypothesis}
\crefname{hypothesis}{Hypothesis}{Hypotheses}
\newsiamthm{claim}{Claim}
\newsiamremark{fact}{Fact}
\crefname{fact}{Fact}{Facts}

\headers{Adaptive Multilevel Newton}{Nick Tsipinakis, Panos Parpas, and Matthias Voigt}

\title{Adaptive Multilevel Newton: A Quadratically Convergent Optimization Method\thanks{Submitted to the editors \today.
}}

\author{Nick Tsipinakis\thanks{Faculty of Mathematics and Computer Science, UniDistance Suisse, Brig, Switzerland  (\email{nikolaos.tsipinakis@unidistance.ch}, \email{matthias.voigt@fernuni.ch}).}
\and Panos Parpas\thanks{Department of Computing, Imperial College, London, UK
  (\email{panos.parpas@imperial.ac.uk}).}
\and Matthias Voigt\footnotemark[2]
}

\usepackage{amsopn}

\newcommand{\R}{\mathbb{R}}
\newcommand{\N}{\mathbb{N}}

\newcommand{\e}{\mathrm{e}}
\renewcommand{\d}{\,\mathrm{d}}
\renewcommand{\epsilon}{\varepsilon}
\usepackage{multirow}

\usepackage{booktabs}
\usepackage{adjustbox}
\usepackage{graphicx}
\usepackage{subcaption}
\usepackage{caption}
\usepackage{float}

\usepackage{mathtools}
\usepackage[shortlabels]{enumitem}

\ifpdf 
\hypersetup{
  pdftitle={Adaptive Multilevel Newton: A Quadratically Convergent Optimization Method},
  pdfauthor={Nick Tsipinakis, Panos Parpas, and Matthias Voigt}
}
\fi

\usepackage{mathtools}

\usepackage{pgfplots}
\usepackage{tikz}
\pgfplotsset{compat=1.18}
\usetikzlibrary{external}
\newcommand{\transp}{\top}
\begin{document}

\maketitle

% REQUIRED
\begin{abstract}
%It is well known that the initial phase of Newton's method may exhibit slower convergence than vanilla Gradient Descent for the class of strongly convex functions. Classical Newton-type multilevel methods improve Newton's performance in this early phase but, like Gradient Descent, achieve only linear convergence near the minimizer. Thus, one typically employs an inexpensive method (e.g., a multilevel scheme or Gradient Descent) in the first phase and switches to Newton's method when the latter has entered its quadratic phase; however, a robust, principled mechanism for this transition is not available  in the literature. We propose an adaptive multilevel Newton-type method and establish its local quadratic convergence rate for strongly convex functions with Lipschitz continuous Hessians and for self-concordant functions. 
%We perform extensive experiments showing that these local quadratic rates are attainable in practice. Despite potentially higher per-iteration cost than classical multilevel approaches, our method is efficient and consistently outperforms Newton's method, Gradient Descent, and the multilevel Newton method, while providing a principled automatic switch to full Newton precisely when it is most efficient. These findings suggest that second-order methods can outperform first-order methods even on problems where Newton's method is initially slow.
%\textbf{PICK ONE, SHORT OR LONG}
Newton's method may exhibit slower convergence than vanilla Gradient Descent in its initial phase on strongly convex problems. Classical Newton-type multilevel methods mitigate this but, like Gradient Descent, achieve only linear convergence near the minimizer. We introduce an adaptive multilevel Newton-type method with a principled automatic switch to full Newton once its quadratic phase is reached. The local quadratic convergence for strongly convex functions with Lipschitz continuous Hessians and for self-concordant functions is established and confirmed empirically. Although per-iteration cost can exceed that of classical multilevel schemes, the method is efficient and consistently outperforms Newton's method, Gradient Descent, and the multilevel Newton method, indicating that second-order methods can outperform first-order methods even when Newton's method is initially slow. The promising empirical results open new avenues for designing reduced-cost second- and high-order methods with extremely fast convergence rates.

\end{abstract}

% REQUIRED
\begin{keywords}
Newton's method, multilevel methods, convex optimization, local quadratic convergence rates.
\end{keywords}

% REQUIRED
\begin{MSCcodes}
90C30, 90C25, 49M15, 65K05
\end{MSCcodes}

\section{Introduction}
Let $f : \R^n \to \R$ be a twice differentiable and strictly convex function. We also assume that it is bounded from below so that the minimizer exists.
Our goal is to solve the unconstrained optimization problem
\begin{equation} \label{eq: problem}
    x^* \coloneq \underset{x \in \R^n}{\operatorname{min}}  \ f(x),
\end{equation}

Among optimization methods applied for solving \eqref{eq: problem}, the pure Newton method has the most powerful theory, exhibiting quadratic convergence rate in a neighborhood of the minimizer.
Once Newton's method enters this neighborhood, convergence is extremely fast, doubling the number of the correct digits at each iteration. In practice this means that highly accurate solutions can be reached in only a few iterations, i.e., three or four are typically enough \cite{boyd2004convex, nesterov2018lectures}. Pure Newton's main drawback is related to the cost of forming the Newton direction and the storage requirements of the Hessian matrix. The per-iteration cost of Newton’s method scales cubically with the problem dimension $n$, which can significantly affect the overall runtime of the algorithm. Nonetheless, this cost is often justified by the method’s rapid convergence once the iterates are sufficiently close to the solution $x^*$. In particular, the local quadratic convergence of Newton’s method is especially advantageous in solving highly ill-conditioned problems where high-precision solutions are required—tasks that may be infeasible for methods limited to linear convergence rates.

In contrast, the global convergence of (damped) Newton's method is less efficient. Specifically, for strongly convex functions, the damped Newton method achieves only a linear convergence rate \cite{karimireddy2018global}. This rate matches that of Gradient Descent for the same class of functions, rendering the higher computational cost of forming and solving systems involving the Hessian matrix unjustified when compared to the substantially cheaper iterations of Gradient Descent. The issue of whether second-order methods are indeed advantageous at the initial stages of optimization was also raised in \cite{nesterov2008accelerating} for the purposes of the accelerated Cubic Newton method. Moreover, empirical observations further support this discrepancy, with Gradient Descent often exhibiting faster progress than second-order methods during the initial stages of the minimization process \cite{bottou2018optimization}. Consequently, due to its significantly lower per-iteration cost, Gradient Descent may substantially outperform Newton's method in terms of runtime when the iterates are far from the minimizer $x^*$. 

To achieve efficient performance in terms of overall runtime, one should then employ an inexpensive method, such as Gradient Descent, during the early stages of optimization, and to switch to Newton's method once the iterates are sufficiently close to the minimizer, where its quadratic rate guarantees the rapid convergence to high precision solutions. However, there are particular challenges when one attempts to combine an inexpensive method with Newton's method that are still open and we try to address in this paper. Our main question is the following: \medskip
\begin{center} 
    \emph{At what point, and according to what criterion, should one transition \\ from a computationally inexpensive method with ``slow'' rate to a more \\ expensive method exhibiting quadratic convergence?} \medskip
\end{center}

Naturally, the more computationally expensive method should be employed during the later stages of the optimization process. A common heuristic is to initiate a second-order method once $\| \nabla f(x) \| < \epsilon$ for some small $\epsilon > 0$. However, selecting an appropriate value for $\epsilon$ is nontrivial and problem-dependent. For instance, in problems that pose difficulties to optimzation methods, a small $\epsilon$ may never be achieved by the inexpensive method, potentially leading to convergence to a suboptimal solution. Conversely, if $\epsilon$ is selected too large, the method performs unnecessary expensive steps. This issue is illustrated in \cref{fig:generated_intro}.

\begin{figure}[bt]
    \centering
    \begin{subfigure}[b]{0.49\textwidth}
     \includegraphics{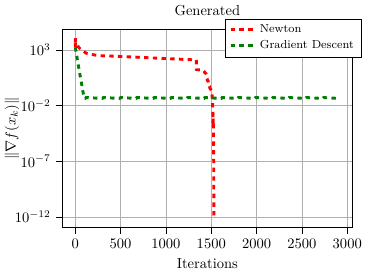}
     \end{subfigure}
    \caption{Minimization of the Poisson loss using a low-rank Generated dataset. Gradient Descent and Newton's methods achieve slow convergence rates in late and early stages of the minimization process, respectively. Gradient Descent significantly outperforms Newton's method in early stages. According to the figure, one must select $\epsilon \geq 10^{-1}$, otherwise the hybrid process may converge to a sub-optimal solution due to Gradient Descent's inefficiency near the minimizer. More details on this experiment can be found in \cref{sec: experiments}.}
    \label{fig:generated_intro}
\end{figure}

\subsection*{Adaptive Multilevel Newton Method}
To the best of our knowledge, a robust hybrid strategy that automatically and adaptively determines a meaningful transition point between methods in a principled manner is not available in the literature.
In this paper we propose an Adaptive Multilevel (AML) Newton method to address this gap in the literature. For the proposed method, we establish a local quadratic convergence rate for both strongly convex functions with Lipschitz continuous  Hessians and self-concordant functions, respectively. To achieve these rates, we introduce a new criterion whereby the proposed method \emph{adaptively} alternates between $m > 1$ subspaces (or \emph{coarse} levels) of $\R^n$, selecting at each iteration the smallest one that provides a theoretically justified and efficient decrease in the norm of the gradients. Specifically, the theory suggests that the \textsc{A\textsc{ML--Newton}} method automatically selects very small coarse levels (i.e., faster iterates) in early stages where the rate of convergence is typically slow, and larger ones near $x^*$ to obtain its quadratic rate. The local quadratic rates are easy to interpret and are comparable to that of Newton's method. As in Newton's method, once the method enters its quadratic phase, it will maintain this rate until the end of the minimization process. Furthermore, we establish a connection with random subspace Newton-type methods. In particular, for the Random Subspace Newton (RSN) method, we show that it achieves a local quadratic rate with high probability, based on random embeddings that satisfy the Johnson–Lindenstrauss (JL) lemma \cite{johnson1984extensions}.

We recall that local quadratic convergence rates are generally not applicable to classical multilevel (\textsc{ML--Newton}) or random subspace Newton-type methods. Without additional structural assumptions on problem \eqref{eq: problem}, one can typically expect only a local linear-quadratic rate \cite{gower2019rsn, hanzely2020stochastic, ho2019newton, tsipinakis2024multilevel}. In this work, we significantly improve the local convergence behavior of these methods through a simple modification: we alternate between subspaces by looking one step ahead in the minimization process, monitoring gradient progress within the relevant subspaces. This strategy ensures local quadratic convergence as the gradient is provably shown to decay more rapidly within the subspaces.

Next, we note that the trial step may introduce additional computational costs at each iteration. However, the actual per-iteration cost of the method remains an open question, as it strongly depends on the specific design of the subspaces. Despite this, we propose a simple yet effective strategy, under which no more than two coarse-level steps are expected on average per iteration. We conduct extensive experiments to evaluate the efficacy of our algorithm in real-world applications. The experiments confirm the theoretical findings, demonstrating local quadratic convergence in all tested cases. Notably, the method can enter the quadratic phase without requiring exact Newton steps, as long as sufficiently large subspaces are employed. Consequently, the behavior of the algorithm near the minimizer closely resembles that of Newton’s method.

Additionally, we compare our approach against Newton's method, Gradient Descent, the \textsc{ML--Newton} method and the subspace Newton method. The proposed method consistently outperforms its counterparts on problems where the structure poses difficulties for both Gradient Descent and Newton’s method, see for instance \cref{fig:generated_intro}. Furthermore, in settings where second-order information is essential and cannot be ignored, our method remains more efficient in terms of runtime due to the lower cost of its early iterations. Moreover, the empirical results between \textsc{AML--Newton} and RSN demonstrate that one can obtain a quadratically convergent second-order method whose per-iteration computational complexity in the first-phase is comparable to that of (already attractive) subspace Newton methods.

\subsection*{Related Work}

The multilevel framework, originally developed for solving linear elliptic partial differential equations (see \cite{borzi2009multigrid} for a review), was brought into optimization to accelerate the convergence of second-order methods such as trust-region and Newton methods. MG/OPT \cite{nash2000multigrid} was the first to extend the PDE multigrid framework to optimization. Building on this, \cite{gratton2008recursive} introduces a fully recursive trust-region scheme (RMTR) across discretization levels with global guarantees, and \cite{gross2009convergence} strengthens the theory with global convergence under milder assumptions on coarse models and transfer operators while allowing inexact subproblem solves. For constrained problems, \cite{ziems2011adaptive} develops an adaptive multilevel inexact SQP framework. In the unconstrained setting,~\cite{ho2019newton} provides convergence guarantees for a multilevel Newton method in the strongly convex case, while \cite{tsipinakis2024multilevel} extends these results to self-concordant functions with bounds that do not depend on unknown problem constants; \cite{tsipinakis2023lowrank} further analyzes a low-rank variant with theory for self-concordant objectives. Moreover, \cite{tsipinakis2025regularized} introduces a multilevel framework for regularized Newton methods in nonconvex optimization. This work is closest to ours: it shows that the multilevel trial steps can match the convergence guarantees of regularized Newton, but its practical efficiency has yet to be verified empirically.

Beyond their well-established theoretical foundations, multilevel methods also deliver efficient performance in modern, real-world applications. Multilevel training for deep learning includes globally convergent RMTR-style methods for ResNets \cite{kopanicakova2022globally}, stochastic MG/OPT for ResNets with reported speedups \cite{von2021training}, multilevel minimization tailored to depth hierarchies \cite{gaedke2021multilevel}, and a stochastic preconditioned gradient method for escaping saddle points and flat areas efficiently \cite{tsipinakis2023simba}. Furthermore \cite{gratton2024block} reframes multilevel as block-coordinate descent, and demonstrates gains for deep-learning PDE surrogates, while \cite{calandra2019approximation} proposes multilevel Levenberg–Marquardt training scheme that optimizes neural networks.
Layer-parallel (MGRIT) approaches expose parallelism across layers for ResNets \cite{gunther2020layer} and GRUs~\cite{moon2022parallel}, while multigrid-style training schedules accelerate video models \cite{wu2020multigrid}. Recent software demonstrates scalable layer-parallel training on MPI/GPU systems \cite{cyr2025torchbraid}, and multilevel-in-width schemes regularize and speed network training \cite{ponce2023multilevel}.

Subspace variants of the Newton method have been at the forefront of research over the past decade. Their main advantage over classical multilevel methods is that they do not require the original Hessian matrix during the optimization and thus they are well-suited for large-scale optimization. Although not directly comparable to the \textsc{A\textsc{ML--Newton}} method proposed in this paper, these approaches are worth mentioning because they can considerably accelerate the otherwise slow initial phase of Newton’s method.

An example of these variants is Randomized Block Coordinate Descent (RBCD) where at each iteration it updates a randomly selected subset of coordinates and therefore requires only the corresponding gradient components. In Newton--type block variants, one additionally uses the associated block of the Hessian, see \cite{nesterov2012efficiency, richtarik2014iteration, wright2015coordinate}. Under standard assumptions (smoothness, strong convexity), Block Coordinate Descent enjoys global convergence with linear rate; the constants are usually better for the Newton variant than for first-order RBCD because the step uses the curvature in the active block.

Random subspace methods generalize block coordinate methods by replacing coordinate-aligned blocks with arbitrary low-dimensional subspaces specified by a sketch matrix $R$. In practice, $R$ often originates from the Johnson--Lindenstrauss (JL) lemma \cite{johnson1984extensions}, which effectively reduces the problem's dimensionality without discarding essential second-order information. The RSN \cite{gower2019rsn} efficiently uses sketching to compute a Newton step restricted to a low-dimensional subspace. SDNA \cite{qu2016sdna} is a stochastic dual Newton method for empirical risk minimization that updates small blocks of dual variables using local Hessian information. The Randomized Block Cubic Newton method \cite{doikov2018randomized} performs cubic regularized Newton updates on randomly sampled variable blocks using block Hessians to capture curvature while keeping iterations cheap. A general randomized subspace framework to tackle large non-convex problems yielding cheaper iterations with global first-order convergence guarantees is proposed in \cite{cartis2022randomised}. The work in \cite{cartis2025random} applies the ideas of the general subspace framework \cite{cartis2022randomised} to Adaptive Regularization with Cubics (ARC) to achieve scalable cubic regularization steps in low-dimensional random subspaces, while preserving ARC’s optimal first- and second-order complexity guarantees.

\subsection*{Contributions}
In light of the related work, we now summarize our contributions.
\begin{enumerate}
    \item We develop \textsc{AML--Newton}, an adaptive multilevel Newton method that forms candidate steps on coarse levels to reduce the cost of full Newton updates. We prove local quadratic convergence rate for the class of strongly convex functions (\cref{thm: rate strong convexity}) and self-concordant functions (\cref{thm: quadratic rate self conc}).  We furter show connections to subspace Newton methods and, Using JL--type results, establish a probabilistic local quadratic convergence rate (\cref{cor: probabilistic rate strong convexity}). The quadratic rates have constants that are comparable to full Newton's method and easy to interpret. 
    \item This paper significantly advances the theory of multilevel methods such as RMTR \cite{gratton2008recursive} and \textsc{ML--Newton} \cite{ho2019newton,tsipinakis2024multilevel}. Our key difference is a new acceptance rule that predicts when a coarse step will be effective, providing a principled mechanism to switch automatically from a subspace method to full Newton steps precisely \emph{when} they are most efficient.
    \item Experiments on real-world problems show that \textsc{AML--Newton} behaves like \textsc{ML--Newton} and RSN far from the solution with negligible trial-step overhead, then automatically switches to full Newton and achieves quadratic convergence matching Newton's local rate. When Newton's initial phase is slow, \textsc{AML--Newton} runs faster than Newton, is typically faster than Gradient Descent while reaching high accuracy, and---unlike \textsc{ML--Newton}/RSN with very small subspaces---can exploit very small coarse models early and still converge to the minimizer.
\end{enumerate}

\section{Notation and Preliminaries}
We say that a symmetric matrix $A \in \R^{n \times n}$ is positive semi-definite if $x^\transp A x \geq 0 $ for all $x \in \R^n$ and we write $A \succeq 0$. Similarly, it is positive definite if $x^\transp A x > 0 $ for all $x \in \R^n \setminus \{0\}$ and we write $A \succ0$. We write $A \succeq B$ (resp. $A \succ B$) if $A-B$ is positive semi-definite (resp. positive definite). Moreover, $I_d$ denotes the $d \times d$ identity matrix. For a function $f : \R^n \rightarrow \R$, the gradient and Hessian at $x \in \R^n$ are denoted by $\nabla f(x)$ and $\nabla^2 f(x)$, respectively. In addition, the standard Euclidean norm is denoted by $\|x\|$. For a matrix $P \in \R^{m \times n}$ with rank $n<m$, we denote by $\sigma_1(P) \geq \sigma_2(P) \geq \cdots \geq \sigma_n(P)$ its singular values as $\sigma_i(P) \coloneq \sqrt{\lambda_i (P^\transp P)}$, where $\lambda_i (P^\transp P) > 0$ is the $i$-th largest eigenvalue of the positive definite matrix $P^\transp P$. Further, we have $\|P\| = \|P^\transp\| = \sigma_1 (P)$, for more details on matrix analysis see \cite{horn2012matrix}.

\subsection{Newton's Method for Strongly Convex and Self-Concordant Functions}
Pure Newton's method is well defined for problem \eqref{eq: problem}. It builds its iterates as follows
\begin{equation*}
    x_{k+1} = x_k - \nabla^2 f(x_k)^{-1} \nabla f(x_k).
\end{equation*}
Let $\mu, L >0$ such that $f$ is $\mu$-strongly convex with $L$-Lipschitz continuous Hessians. Then, pure Newton's steps achieve a local quadratic convergence \cite{boyd2004convex} with rate given by
\begin{equation} \label{ineq: newton's quadratic rate}
            \| \nabla f(x_{k+1}) \| \leq \frac{L}{2 \mu^2 } \|\nabla f(x_{k}) \|^2.
\end{equation}
Moreover, a function is called self-concordant with constant $M>0$ if for all $u \in \R^n$,
\begin{equation} \label{ineq: def self concordance}
    |D^3 f(x)[u,u,u]| \leq 2 M \left[u^\transp \nabla^2 f(x) u \right]^{3/2}.    
\end{equation}
Let us fix an $x \in \R^n$. Self-concordance gives rise to the following local norms induced by the Hessian at $x$,
\begin{equation} \label{eq: local norms definition}
    \| y \|_x = \left( y^\transp \nabla^2f(x) y \right)^{\frac{1}{2}} \quad \text{and}  \quad  \| y \|_x^* = \left( y^\transp \nabla^2f(x)^{-1} y \right)^{\frac{1}{2}}
\end{equation}
for any $y \in \R^n$. The Newton's method's local convergence for the class of self-concordant functions is measured based on the local norm of the gradient, or otherwise called the Newton decrement, which is defined as
\begin{align} 
    \lambda(x) & \coloneq \left[\nabla f(x)^\transp \nabla^2 f(x)^{-1} \nabla f(x)\right]^{1/2}. \label{eq: newton decrement}
\end{align}
It can be shown that pure Newton's method converges quadratically for the class of self-concordant functions. In particular, if $M \lambda(x_{k}) < 1$, then
\begin{equation} \label{ineq: newton quadratic phase self concor}
\lambda(x_{k+1}) \leq \frac{M \lambda(x_{k})^2}{\left( 1-M \lambda(x_{k})\right)^2}.
\end{equation}
If $M=1$ in \eqref{ineq: def self concordance}, the function $f$ is called standard self-concordant. This class is important because it guarantees convergence rates that are independent of unknown problem constants. For more details on the convergence rate of Newton's method under strong convexity and self-concordance we refer the reader to the books \cite{boyd2004convex, nesterov2018lectures}.

\subsection{The Classical Multilevel Framework for Newton's Method} \label{sec: classical multilevel methods}
The general multilevel approach for unconstrained optimization has been explored in various settings, see for instance \cite{gratton2008recursive, nash2000multigrid, MR2587737}. Subsequent works studied it as a two-level method, arising naturally from the structure of Newton's method when constructing its iterates \cite{ho2019newton, tsipinakis2024multilevel}. In this section, we unify these perspectives and present a general framework that incorporates a finite hierarchy of subspaces, tailored specifically to Newton-type methods. We refer to this approach as the Multilevel Newton (\textsc{ML--Newton}) method.

To apply the multilevel framework for problem \eqref{eq: problem}, we assume that we have access to $m \in \N$ models. The first $m-1$ models are defined over lower-dimensional subspaces of $\R^n$, whilst the model $m$ corresponds to the objective function $f$. Formally, we define functions
$f_{i} : \R^{n_i} \to \R$, $i = 1, \ldots, m$, where $n_{i} < n_{i+1}$ (note that $f_m \coloneq f$ and $n_m \coloneq n$). Following the standard multilevel terminology, we refer to the levels $i$, with $n_i < n$, as  \emph{coarse} levels. In this setting, multilevel methods construct a hierarchy of models associated with the coarse levels, where the main idea is to exploit the ``small'' models, instead of the more costly $f$, to generate sequences $(x_k)_{k \in \N}$ that converge to $x^*$. In what follows we describe how to generate such sequences.

Our first objective is to propagate the information toward the coarser levels. To access a model $f_i$, $1\leq i \leq m-1$, from the uppermost, or using the multilevel terminology, \emph{fine}, level and vice versa, we make use of the following linear mappings: 
$\tilde{R}_i : \R^{n_i} \to \R^{n_{i-1}}$ and $\tilde{P}_i : \R^{n_{i-1}} \to \R^{n_{i}}$, respectively.
We assume that these linear mappings are surjective and injective, respectively,  i.e., the corresponding linear operators have full rank. Furthermore, we assume that $\tilde{R}_i = \tilde{P}_i^\transp$. Thus, given $x \in \R^n$, we move to $f_i$ with input $x_i = \tilde{R}_{i + 1} \cdots \tilde{R}_{m-1} \tilde{R}_{m} x \in \R^{n_i}$, and back to $f$ with input $x_m = (\tilde{R}_{i + 1} \cdots \tilde{R}_{m-1} \tilde{R}_{m})^\transp x_i = \tilde{P}_{m} \tilde{P}_{m - 1} \cdots \tilde{P}_{i + 1} x_i \in \R^n$. In the multilevel literature, $\tilde{R}_i$ is known as the restriction operator, whereas $\tilde{P}_i$ is the prolongation operator. 

Given the above process of moving to the available levels, we generate sequences $(x_k)_{k \in \N}$ in a natural way, using the second-order Taylor approximation of $f_i$. When $i=m$, this is straight-forward, i.e., at an iteration $k$, one obtains the Newton step. However, if $1 \leq i \leq m-1$, the derivatives 
the low-dimensional model $f_i$ do not match those of the original model $f$. To address this discrepancy, multilevel methods enforce first- and second-order coherency between $f_i$ and $f$ by adding a correction term to $f_i$. In particular, let $k \in \N$ and assume that we are given $x_{i, k}$, a point at the coarse level $i$ at iteration $k$. We then proceed to the next (coarser) level with initial point $x_{i-1, 0} \coloneq \tilde{R}_{i} x_{i, k}$. We define $\psi_{i-1} : \R^{n_{i-1}} \to \R$ such that
\begin{align*} \label{eq:definition-coarse-model}
    \psi_{i-1}(x_{i-1}) \coloneq f_{i-1} (x_{i-1}) + \langle v_{{i-1}, k}, x_{i-1} - x_{{i-1}, 0} \rangle,
\end{align*}
where $v_{i-1,k} \coloneq \tilde{R}_i \nabla f_i(x_{i,k}) - \nabla f_{i-1}(x_{i-1, 0})$ and $2 \leq i \leq m$. Then, with the above modification on $f_{i-1}$ one now ensures that
\begin{equation} \label{eq:first-order-coherency}
        \nabla \psi_{i-1}(x_{i-1, 0}) = \tilde{R}_i \nabla f_i(x_{i,k}),
\end{equation}
i.e., the model $\psi_{i-1}$ is \emph{first-order coherent} w.r.t. the model $f_i$, see also \cite{nash2000multigrid}. To ensure its \emph{second-order coherency}, we define $f_{i-1}$ as 
\begin{equation*} \label{eq:def-f_{i-1}}
    f_{i-1} (x_{i-1}) = \frac{1}{2} \left\langle \tilde{R}_i \nabla^2 f(x_{i, k}) \tilde{P}_i (x_{i-1} - x_{i-1, 0}), x_{i-1} - x_{i-1, 0}\right\rangle,
\end{equation*}
and thus,
\begin{equation} \label{eq:second-order-coherency}
    \nabla^2 \psi_{i-1}(x_{i-1, 0}) = \tilde{R}_i \nabla^2 f_i(x_{i,k}) \tilde{P}_i.
\end{equation}
The choice of $f_{i-1}$ in \eqref{eq:first-order-coherency} is standard in multilevel methods, and is referred to as a \emph{Galerkin model}, originally introduced in \cite{MR2587737}. The work in \cite{MR2587737} and subsequent ones, e.g., \cite{ho2019newton, tsipinakis2024multilevel}, have extensively evaluated its performance, demonstrating its effectiveness in practical applications. Within the multilevel framework, the model $\psi_{i-1}$ is referred to as \emph{coarse}, since $n_i < n$, whereas $f_m$ (or equivalently $f$) is termed the \emph{fine model}.

Since the coherency properties have been established, we can form the second-order Taylor approximation $T_{i-1,k}$ of $\psi_{i-1}$ at $x_{i-1, 0}$, where $2 \leq i \leq m$. Denote $d_{i-1} \coloneq x_i - x_{i-1, 0}$. We have
\begin{equation} \label{eq:Taylor-approx}
\begin{split}
    T_{i-1,k}(x_{i-1, 0} + d_{i-1}) & =   \psi_{i-1}(x_{i-1, 0}) + \langle \nabla \psi_{i-1}(x_{i-1, 0}), d_{i-1} \rangle  \\  & \qquad + \frac{1}{2} \left\langle \nabla^2 \psi_{i-1}(x_k) d_{i-1}, d_{i-1} \right\rangle \\
    & \overset{\eqref{eq:first-order-coherency}, \eqref{eq:second-order-coherency}}{=} \psi_{i-1}(x_{i-1, 0}) + \left\langle \tilde{R}_i \nabla f_i(x_{i,k}), d_{i-1} \right\rangle 
    \\ & \qquad + \frac{1}{2} \left\langle \tilde{R}_i \nabla^2 f_i(x_{i,k}) \tilde{P}_i d_{i-1}, d_{i-1} \right\rangle.
    \end{split}
\end{equation}
Minimizing \eqref{eq:Taylor-approx} over $d_{i-1}$, we obtain the coarse direction at level $i-1$ as
\begin{equation*} 
    d_{i-1, k} = - \left( \tilde{R}_i \nabla^2 f_i(x_{i,k}) \tilde{P}_i \right)^{-1} \tilde{R}_i \nabla f_i(x_{i,k}),
\end{equation*}
and, as usual, we obtain the coarse direction at the next less coarse level by applying the prolongation operator, i.e.,
\begin{equation*} 
    d_{i, k} = \tilde{P}_i d_{i-1, k}.
\end{equation*}
To conclude this section, it remains to discuss a procedure for selecting an appropriate level, or coarse model, based on which the point $x_{k+1}$ will be generated. Ideally, one would select $\psi_1$ to obtain $d_{1, k}$, since $n_1 \leq n_i$ for all $i \geq 2 $ and thus, $d_{1, k}$ is more efficient to compute. However, coarse directions may not always be effective. This may occur when $\tilde{R}_i \nabla f_i(x_{i,k}) = 0$ while $\nabla f_i(x_{i,k}) \neq 0$. This is possible since it might be that $ \nabla f_i(x_{i,k}) \in \operatorname{null}\big({\tilde{R}_i}\big)$. Classical multilevel methods evaluate the effectiveness of a coarse model $\psi_{i-1}$ before computing the associated coarse direction. This assessment can be carried out efficiently by checking the following condition (see also \cite{gratton2008recursive}): let $\gamma \in\big(0, \min\big\{1, \min_i \big\|\tilde{R}_i \big\| \big\}\big)$ and $\epsilon \in (0, 1)$ --- then the coarse model $\psi_{i-1}$ is selected if
\begin{equation} \label{ineq:conditions-coarse-model-classical}
    \big\|\tilde{R}_i \nabla f_i(x_{i,k})\big\| \geq \gamma \| \nabla f_i(x_{i,k})\| \quad \text{and} \quad \big\|\tilde{R}_i \nabla f_i(x_{i,k})\big\| > \epsilon.
\end{equation}
The conditions in \eqref{ineq:conditions-coarse-model-classical} ensure that $\tilde{R}_i \nabla f_i(x_{i,k})$ is nonzero and that its norm is sufficiently close to the true gradient norm, scaled by a user-defined parameter $\gamma > 0$. In multilevel methods, $\gamma$ is a hyperparameter that controls the trade-off between computational efficiency and the quality of the coarse direction. Choosing a smaller $\gamma$ allows the selection of coarser models, which are computationally cheaper; however, the resulting coarse directions may be less informative, as $\tilde{R}_i \nabla f_i(x_{i,k})$ could be small in magnitude relative to $\nabla f_i(x_{i,k})$. Consequently, selecting an appropriate value for $\gamma$ requires careful tuning, as it is inherently problem-dependent. Moreover, the condition $\big\|\tilde{R}_i \nabla f_i(x_{i,k})\big\| > \epsilon$ ensures that the Newton step will always be performed once $\big\|\tilde{R}_i \nabla f_i(x_{i,k})\big\| \leq \epsilon$. Effectively, this condition is useful when high accuracy is required, which is possible to attain due to Newton's quadratic rate.

An efficient implementation of the multilevel strategy proceeds as follows: starting from the coarsest level $i = 1$, one moves sequentially to finer levels until the condition \eqref{ineq:conditions-coarse-model-classical} is met. If a level $1 \leq i \leq m-1$ is selected, we follow the procedure outlined earlier: we compute the initial point $x_{i-1,0} = \tilde{R}_{i} \tilde{R}_{i+1} \cdots \tilde{R}_{m} x_{m,k}$, construct the model $\psi_{i-1}$, and evaluate the corresponding coarse direction $d_{i-1,k}$. We then propagate $d_{i-1,k}$ back to the original level
\begin{equation*}
    d_{m,k} = \tilde{P}_{m} \cdots \tilde{P}_{i+1} \tilde{P}_{i} d_{i-1 ,k}
\end{equation*}
and compute 
\begin{equation*}
    x_{k+1} = x_k + d_{m,k}.
\end{equation*}
On the other hand, if \eqref{ineq:conditions-coarse-model-classical} is not satisfied for $1 \leq i \leq m-1$, we then compute $x_{k+1}$ from the Newton method. 

\section{An Efficient Newton-type Multilevel Method} \label{sec: new multilevel}

The multilevel Newton method presented in the previous section chooses the coarse level for constructing the associated coarse direction according to \eqref{ineq:conditions-coarse-model-classical}. Although, convergence to the unique minimizer can now be established due to \eqref{ineq:conditions-coarse-model-classical}, the local rates for multilevel methods, whether Newton-based or otherwise, are generally linear-quadratic. Indeed, without additional assumptions, linear-quadratic rates are the best one can typically expect. 

In this section we introduce a simple, yet effective, idea to significantly enhance the local convergence behavior of the \textsc{ML--Newton} methods. The key insight is to compute the next iterate by looking one step ahead. In what follows, we present how the new multilevel algorithm generates the sequence $(x_k)_{k \in \N}$. We also present a modified multilevel framework for the Newton-based method. Compared to the classical framework in \cref{sec: classical multilevel methods}, this framework is more general, offering more flexibility for practical applications.

In the classical multilevel framework in \cref{sec: classical multilevel methods}, reaching a coarse level $i$, $1 \leq i \leq m-1$, requires passing through every level in the hierarchy. Our approach here allows direct access to any chosen coarse level. This flexibility enables the algorithm to exploit information from a specific coarse space immediately, potentially reducing unnecessary intermediate computations. Unless stated otherwise, we follow the notation introduced in \cref{sec: classical multilevel methods}.

Let us fix some $1\leq i \leq m-1$. As previously, our first objective is to propagate the information from the original toward the coarse level $i$ and vice versa. In contrast to \cref{sec: classical multilevel methods}, we give the following definition for the linear mappings: $R_i : \R^{n} \to \R^{n_i}$ and $P_i : \R^{n_i} \to \R^{n}$, respectively. Similarly, we assume that (i) the corresponding linear operators have rank $n_i$, (ii) $n_{i} < n_{i+1}$, and (iii) $R_i = P_i^\transp$.  Therefore, we now obtain $ x_i = R_i x$ and $x_m = P_i x_i$, where the subscript $i$ in $x_i$ is used since $x_i \in \R^{n_i}$. The coarse model $\psi_i : \R^{n_i} \to \R$ is defined in a similar fashion. Let $k \in \N$ and denote $x_{i, 0} \coloneq R_i x_k$, i.e., the initial point at level $i$. Then,
\begin{align*} \label{eq:definition-coarse-model modified}
    \psi_{i}(x_{i}) \coloneq f_{i} (x_{i}) + \langle u_{{i}, k}, x_{i} - x_{{i}, 0} \rangle,
\end{align*}
where $u_{i,k} \coloneq R_i \nabla f(x_{k}) - \nabla f_{i}(x_{i, 0})$. Defining $f_i$ as
\begin{equation*} \label{eq:def-f_i modified}
    f_{i} (x_{i}) = \frac{1}{2} \langle R_i \nabla^2 f(x_{k}) P_i (x_{i} - x_{i, 0}), x_{i} - x_{i, 0}\rangle,
\end{equation*}
the coarse model is both first- and second-order coherent w.r.t. $f$.
Moreover, from the second-order Taylor approximation of $\psi_i$ at $x_{i, 0}$ and setting $d_{i} \coloneq x_{i} - x_{i, 0}$, we obtain the second-order Taylor polynomial $T_{i,k}$ as 
\begin{equation} \label{eq:Taylor-approx modified}
\begin{split}
    T_{i,k}(x_{i, 0} + d_{i}) & = \psi_{i}(x_{i, 0}) + \langle R_i \nabla f(x_{k}), d_{i} \rangle + \frac{1}{2} \langle R_i \nabla^2 f(x_{k}) P_i d_{i}, d_{i} \rangle.
    \end{split}
\end{equation}
Minimizing \eqref{eq:Taylor-approx modified} over $d_i$ we get the coarse direction $d_{i, k} \in \R^{n_i}$
\begin{equation} \label{eq: coarse direction at i}
    d_{i,k} = - \left(R_i \nabla^2 f(x_k) P_i \right)^{-1} R_i \nabla f(x_k),
\end{equation}
which is then prolongated to obtain the coarse direction $d_{m.k} \in \R^n$ as 
\begin{equation} \label{eq: coarse direction at m}
    d_{m, k} = P_i d_{i, k} = - P_i \left(R_i \nabla^2 f(x_k) P_i \right)^{-1} R_i \nabla f(x_k).
\end{equation}
We note that the multilevel setting in this section is more general than the classical one of \cref{sec: classical multilevel methods}. One can retrieve the classical multilevel framework by setting $R_i \coloneq \tilde{R}_{i + 1} \cdots \tilde{R}_{m-1} \tilde{R}_{m}$. 

The derivation of the coarse step in \eqref{eq: coarse direction at m} is straight-forward, and naturally arises from the classical framework in \cref{sec: classical multilevel methods}. The main novelty in this section are the new conditions we introduce below for selecting the coarse model, where, as we will see later, they yield more effective coarse directions. Let $\sigma \in (0, \min\{1, \min_i \|R_i\| \})$. Then, the coarse direction \eqref{eq: coarse direction at m} arising from the level $i$ is selected if 
\begin{equation} \label{ineq:conditions-coarse-model-new}
    \|R_i \nabla f(x_{k+1})\| \geq \sigma \| \nabla f(x_{k+1})\|.
\end{equation}
In view of the new condition, we notice that the main difference between \eqref{ineq:conditions-coarse-model-new} and \eqref{ineq:conditions-coarse-model-classical} is that we now require $R_i \nabla f(x_{k+1})$ to be close to the true gradient $\nabla f(x_{k+1})$. In particular, at some iteration $k$, we choose the coarse model $\psi_i$ which yields $d_{m,k}$ if, at iteration $k+1$, $R_i \nabla f(x_{k+1})$ is close compared to $\nabla f(x_{k+1})$. In other words, we seek coarse directions whose benefit is realized in the subsequent iteration rather than the current one. We term the resulting multilevel Newton-type method \emph{adaptive} because it dynamically selects the coarse levels to accelerate convergence. Surprisingly enough, we will see that one can now prove a local quadratic convergence rate for the new multilevel algorithm.

As a result, if for our $i$, relationship \eqref{ineq:conditions-coarse-model-new} holds, we compute
\begin{equation} \label{eq: coarse step}
    x_{k+1} = x_k + d_{m, k},
\end{equation}
otherwise we check condition \cref{ineq:conditions-coarse-model-new} for $i+1$.
If there exists no $1 \leq i \leq m-1$ such that \eqref{ineq:conditions-coarse-model-new} is satisfied, then we set $i=m$ and $R_m = I_n$ and perform pure Newton steps, i.e.,
\begin{equation} \label{eq: newton step}
    x_{k+1} = x_k + d_{k},
\end{equation}
where $d_k \coloneq - \nabla^2 f(x_k)^{-1} \nabla f(x_k)$ is the Newton direction.

\begin{algorithm}[t]
\caption{Adaptive Multilevel Newton method}
\begin{algorithmic}[1]
\REQUIRE initial value $x_{0} \in \R^n$, restriction operators $R_{i, k}$, damping parameters $\alpha_k, \sigma \in (0,1]$
\ENSURE sequence $(x_k)_{k \in \N}$ converging to $x^*$
\FOR{$k = 0, 1, \ldots$}
% \STATE Find the smallest $1 \leq i \leq m$ that satisfies $\|R_{i,k} \nabla f(x_{k+1})\| \geq \sigma \| \nabla f(x_{k+1})\|$ \label{step: find}
\STATE Compute \begin{equation*} \label{step: find}
i^*:=\min\left\{\,1\le i\le m:\;
\begin{aligned}
\|R_{i,k}\nabla f(x_k)\| &\ge \sigma\|\nabla f(x_k)\|, \\
\|R_{i,k}\nabla f(x_{k+1})\| &\ge \sigma\|\nabla f(x_{k+1})\|
\end{aligned}
\right\}.
\end{equation*}

\IF {$1 \leq i^* \leq m-1$}
\STATE Set $ x_{k+1} := x_{k} - \alpha_k P_{i^*, k} \left(R_{i^*, k} \nabla^2 f(x_k) P_{i^*, k} \right)^{-1} R_{i^*, k} \nabla f(x_k)$.
\ELSE
\STATE Set $ x_{k+1} := x_{k} - \alpha_k \nabla^2 f(x_k)^{-1} \nabla f(x_k)$.
\ENDIF
\ENDFOR
\end{algorithmic}
\label{alg: multilevel}
\end{algorithm}

Furthermore, one could define iteration-dependent restriction and prolongation operators, i.e., $R_{i,k}$ and $P_{i,k}$, respectively, and derive the algorithm's coarse steps without modifying the above analysis. A generic Adaptive Multilevel Newton method (\textsc{A\textsc{ML--Newton}}) with iteration-dependent operators {and damping parameter $\alpha_k$}
is presented in \cref{alg: multilevel}. 
Note that, when applying iteration-dependent operators, \cref{alg: multilevel} must also check the condition
$\|R_{i,k}\nabla f(x_k)\|\ge \sigma\|\nabla f(x_k)\|$ at each iteration. We show below that, in the
classical multilevel setting with fixed operators, this condition follows directly from
\cref{ineq:conditions-coarse-model-new} and can therefore be removed from step \ref{step: find}.
There are several ways to compute the damping parameter at each iteration.
The Armijo rule is often preferred over other strategies in second-order methods because it is robust and automatically adapts the step size, accepting the full Newton step when it is safe.

A process for computing $i^*$ in step \ref{step: find} in \cref{alg: multilevel} will be discussed in \cref{sec: implementation details}. In any case though, one has to compute a trial step to evaluate \eqref{ineq:conditions-coarse-model-new}. This may come with the cost of constructing extra coarse directions at each iteration with cost $\mathcal{O}\big(n_i^3\big)$, and thus, per iteration, \cref{alg: multilevel} is more expensive than the \textsc{ML--Newton} method. 
The total complexity of \cref{alg: multilevel} is difficult to measure since it is very much dependent on the design of coarse models and the choice of $\sigma$. In the worst scenario though the method will compute $m-1$ coarse steps plus a Newton step. In \cref{sec: implementation details} we will discuss how to design coarse models such that in practice at most two coarse steps with $n_i \ll n$ are computed on average, at each iteration.  

We also note that the idea of taking a trial step is not new in optimization. It has been a key component in trust-region methods \cite{conn2000trust}. Nevertheless, although the main idea looks the same, it is in fact different from that of \cref{alg: multilevel}. In trust-region methods, one looks at the ratio of the actual vs. predicted decrease of the objective function to decide whether a candidate step is accepted. Then, the trust-region radius is adjusted accordingly. Here, instead, we look at the ratio of gradient norms to, in essence, determine whether a level is accepted. As we will see later, the norm of the reduced gradient decreases faster than that of the true gradient, thus our goal is to find a level that imposes a similar decrease to the latter norm.

\section{Convergence Analysis} \label{sec: analysis}

{We start by discussing the global convergence of \cref{alg: multilevel}. To show global convergence we consider two cases: (i) the restriction operator is fixed at each iteration, i.e, $R_{i,k} \equiv R_{i}$, (ii) it varies as in \cref{alg: multilevel}}.

Let $R_i$, $1\leq i\leq m-1$, be fixed and assume for simplicity that the coarse step is always accepted (if no coarse model accepted at an iteration $k$, then the method progresses as Newton's method whose global convergence is already established). Then by inequality \cref{ineq:conditions-coarse-model-new} there exists a $j \in \{1,\ldots,m-1\}$ with 
\begin{equation*}
    \|R_j \nabla f(x_{k})\| \geq \sigma \| \nabla f(x_{k})\|. 
\end{equation*}
Thus, the method takes effective steps at each iteration similar to classical Newton multilevel method (see condition \eqref{ineq:conditions-coarse-model-classical}). 
Then, assuming that the function is strongly convex with Lipschitz continuous gradients and Hessians the required global result follows from \cite{ho2019newton}. It can be shown that 
\begin{equation} \label{ineq: global convergence}
    f(x_{k+1}) \leq f(x_k) - \nu_k,
\end{equation}
where $\nu_k \coloneq \sigma C \|\nabla f(x_k)\|^2$ for some $C>0$ and 
for all $k \in \N$ and $x_0 \in \R^n$. Specifically,
the authors show that \eqref{ineq: global convergence} remains true even when the damping parameter $\alpha_k$ arises from the Armijo rule \cite[Lemma 3.1]{ho2019newton} and that the Armijo rule accepts the unit step \cite[Lemma 3.8]{ho2019newton}. Analogously, we can show similar global results for \cref{alg: multilevel} when assuming strictly convex self-concordant functions using Lemmas 3.1 and 3.4 proved in \cite{tsipinakis2024multilevel}. In this case $\nu_k$ depends on the Newton decrement instead of norm of the gradient (see \cite{tsipinakis2024multilevel}).

If \cref{alg: multilevel} is applied with iteration-dependent operators, then \cref{ineq:conditions-coarse-model-new} does not necessarily imply $    \|R_{i, k} \nabla f(x_{k})\| \geq \sigma \| \nabla f(x_{k})\|$ anymore, as $\|R_{i, k} \nabla f(x_{k})\|$ can be arbitrarily small compared to $\|\nabla f(x_{k})\|$. Thus, in order to establish the global convergence of the algorithm we need to check $    \|R_{i, k} \nabla f(x_{k})\| \geq \sigma \| \nabla f(x_{k})\| $ at each iteration alongside \cref{ineq:conditions-coarse-model-new}. Then, as previously, for both strongly convex and self-concordant functions an inequality such as \cref{ineq: global convergence} can be shown from \cite{ho2019newton} and \cite{tsipinakis2024multilevel}. 

The above discussion shows that the global convergence behavior of the multilevel Newton methods for strictly convex functions aligns with that of the classical Newton method, (similar reductions in the value of the objective function have been shown in \cite{nesterov2018lectures, boyd2004convex}).
However, their local convergence analyses remain weaker compared to the well-established local behavior of the classical Newton method. In this section, we fill this gap by presenting a local convergence analysis of Algorithm~\ref{alg: multilevel}, establishing its quadratic rate under two settings: 
\begin{enumerate}[(i)]
  \item strongly convex functions with Lipschitz-continuous Hessians;
  \item strictly convex self-concordant functions.
\end{enumerate}
In this section we simplify the notation as follows: the gradient and the Hessian at iteration $k$ are denoted $\nabla f_k \coloneq \nabla f(x_k)$ and $\nabla^2 f_k \coloneq \nabla^2 f(x_k)$, respectively.  
We will perform an analysis of \cref{alg: multilevel} in more generic form that allows for $R_i$ to vary at each iteration. The operator $R_{i, k}$ denotes the restriction operator at level $i$ and iteration $k$. Therefore, for the purposes of this section, all quantities of \cref{sec: new multilevel} involving $R_i$ should be replaced with $R_{i, k}$.  We moreover define
\begin{equation} \label{eq: def G_k}
    G_k \coloneq \int_0^1 \nabla^2 f(x_k + t(x_{k+1} - x_k)) \d t,
\end{equation}
where $x_k$ arises from \cref{alg: multilevel}. 
For the remainder of this section, to derive the local convergence rates of \cref{alg: multilevel}, we set $\alpha_k \equiv 1$ By the discussion above, if the current iterate is sufficiently close to the minimizer, the Armijo rule will always accept the unit step.

\subsection{Strongly Convex Functions} \label{sec: strong convexity analysis}
In this section we analyze \cref{alg: multilevel} for strongly convex functions with Lipschitz Hessians \cite{boyd2004convex, nesterov2018lectures}. That is, we assume there exist constants $\mu, L > 0$ such that
\begin{equation} \label{ineq: strong convexity}
    \nabla^2 f(x) \succeq \mu I_n \quad \text{for all } x \in \R^n,
\end{equation}
and
\begin{equation} \label{ineq: Lipschitz hessian}
    \big\| \nabla^2 f(x) - \nabla^2 f(y)\big\| \leq L \| x-y \| \quad \text{for all } x, y \in \R^n.
\end{equation}

Let $\sigma_1(P_{i,k}) \geq \cdots \geq \sigma_{n_i} (P_{i,k}) >0$ be the singular values of $P_{i, k}$. Then, by the properties of the Euclidean norm, we have
\begin{equation} \label{eq: def norm R_i,k}
    \| R_{i, k}\| = \| P_{i, k}\| = \sigma_1(P_{i,k}).
\end{equation}
We further assume that the sequence $(P_{i, k})_{k \in \N}$ is bounded in the sense that there exist $0< \omega, \xi < \infty$ such that
\begin{equation} \label{eq: def omega and xi}
    \omega \coloneq \inf_{i, k} \{ \sigma_{n_i} (P_{i,k}) \} \quad \text{and} \quad \xi \coloneq \sup_{i, k} \{ \sigma_1(P_{i,k})\}.
\end{equation}
We will see later that there are efficient ways to generate $(P_{i, k})_{k \in \N}$ to satisfy \eqref{eq: def omega and xi} for such $\omega$ and $\xi$. 
Furthermore, from \eqref{eq: def norm R_i,k} and \eqref{eq: def omega and xi} we have
\begin{equation} \label{ineq: norm R_i, omega and xi}
    \| R_{i, k}\| \leq \xi \quad \text{and} \quad \omega \leq \xi.
\end{equation}
The following result plays a key role in establishing the local quadratic convergence of \cref{alg: multilevel}.
\begin{lemma} \label{lemma: strongly convex rate subspace}
If the point $x_{k+1}$ is obtained by the coarse model as in \eqref{eq: coarse step}, then 
    \begin{equation} \label{ineq: strongly convex rate subspace}
        \|R_{i,k} \nabla f_{k+1} \| \leq \frac{L \xi^5}{2 \mu^2 \omega^4} \|\nabla f_{k} \|^2
    \end{equation}
    for all $i \in \{1,2, \ldots, m-1 \}$ and $k \in \N$.
\end{lemma}

\begin{proof}
Let us fix an $i \in \{1,2, \ldots, m-1 \}$ and $k \in \N$. From Taylor's identity $\nabla f_{k+1} = \nabla f_{k} + G_k (x_{k+1} - x_k)$ we have
    \begin{align}
        R_{i, k} \nabla f_{k+1} & \overset{\eqref{eq: coarse step}}{=} R_{i, k} \nabla f_{k} + R_{i, k} G_k d_{m,k} \notag \\
        & \overset{\eqref{eq: coarse direction at i}}{=} - R_{i,k} \nabla^2 f_{k} P_{i, k} d_{i,k} + R_{i, k} G_k d_{m,k} \notag \\
        & \overset{\eqref{eq: coarse direction at m}}{=} - R_{i,k} \nabla^2 f_{k} d_{m,k} + R_{i, k} G_k d_{m,k} \notag \\
        & = -R_{i,k} \left( \nabla^2 f_{k} - G_k \right) d_{m,k}. \label{eq: identity from taylor formula}
    \end{align}
    Moreover, using the Singular Value Decomposition of $P_{i,k}$, and since $P_{i,k}$ has full rank, we have that
    \begin{equation} \label{ineq: bounds on RP}
        R_{i, k} P_{i, k} \succeq \sigma_{n_i}^2 I_n \overset{\eqref{eq: def omega and xi}}{\succeq} \omega^2 I_n \implies \left(R_{i, k} P_{i, k}\right)^{-1} \preceq \frac{1}{\omega^2}I_n.
    \end{equation}
    Furthermore, 
    \begin{equation} \label{ineq: integral lipschitz}
        \begin{split}
            \big\| \nabla^2 f_{k} - G_k \big\| & \leq \int_0^1 \big\| \nabla^2 f_{k} - \nabla^2 f(x_k + td_{m, k}) \big\| \d t \\
            & \overset{\eqref{ineq: Lipschitz hessian}}{\leq} \int_0^1 t  \d t \ L \| d_{m, k}\| = \frac{L}{2} \| d_{m, k}\|.
        \end{split}
    \end{equation}
    Let us now estimate an upper bound on $\|d_{m,k}\|$. First note that by strong convexity we obtain
    \begin{align} 
            \nabla^2 f(x) \succeq \mu I_n & \implies R_{i, k}\nabla^2 f(x) P_{i, k} \succeq \mu R_{i, k}P_{i, k}  \notag \\ 
            & \implies \left( R_{i, k}\nabla^2 f_k P_{i, k}\right)^{-1} \preceq \frac{1}{\mu} \left( R_{i, k}P_{i, k}\right)^{-1} \overset{\eqref{ineq: bounds on RP}}{\preceq} \frac{1}{\omega^2 \mu} I_n, \label{ineq: upper-bound on RnablafP}
    \end{align}
    and thus
    \begin{align}
            \| d_{m, k}\| & \overset{\eqref{eq: coarse direction at m}}{=} \left\| P_{i, k} \left( R_{i, k}\nabla^2 f_k P_{i, k}\right)^{-1} R_{i, k}\nabla f_k \right\| \notag \\ 
            & \overset{\eqref{ineq: norm R_i, omega and xi}}{\leq} \xi \left\|\left( R_{i, k}\nabla^2 f_k P_{i, k}\right)^{-1} R_{i, k}\nabla f_k \right\| \notag \\
            & \overset{\eqref{ineq: upper-bound on RnablafP}}{\leq} \frac{\xi}{\mu \omega^2} \| R_{i, k}\nabla f_k \|. \label{ineq: upper bound on d_m} 
            % & \overset{\eqref{ineq: norm R_i, omega and xi}}{\leq} \frac{\xi^2}{\mu \omega^2} \| \nabla f_k \|. 
    \end{align}
    Now we have
    \begin{align}
        \| R_{i, k} \nabla f_{k+1} \| & \overset{\eqref{eq: identity from taylor formula}}{=} \big\| R_{i,k} \left( \nabla^2 f_{k} - G_k \right) d_{m,k}\big\| \notag \\
        & \leq \| R_{i,k} \| \big\| \nabla^2 f_{k} - G_k \big\| \| d_{m,k} \| \notag \\
        & \overset{\eqref{ineq: norm R_i, omega and xi}}{\leq} \xi \| \nabla^2 f_{k} - G_k \| \| d_{m,k} \| \notag \\ 
        & \overset{\eqref{ineq: integral lipschitz}}{\leq} \frac{L \xi}{2} \| d_{m, k}\|^2 \notag \\
        & \overset{\eqref{ineq: upper bound on d_m}}{\leq} \frac{L \xi^3}{2 \mu^2 \omega^4} \| R_{i,k} \nabla f_k \|^2  \label{ineq: upper step before} \\
        & \overset{\eqref{ineq: norm R_i, omega and xi}}{\leq} \frac{L \xi^5}{2 \mu^2 \omega^4} \| \nabla f_k \|^2, \notag 
    \end{align}
    as required.
\end{proof}

We are now in position to show the main result for \cref{alg: multilevel}. Recall the conditions introduced in \eqref{ineq:conditions-coarse-model-new} for iteration-dependent operators, i.e.,
\begin{equation} \label{ineq:conditions-coarse-model-new R_k}
    \|R_{i, k} \nabla f_{k+1}\| \geq \sigma \| \nabla f_{k+1}\|.
\end{equation}
% It is now straight-forward to obtain the local quadratic convergence rate.
\begin{theorem} \label{thm: rate strong convexity}
    Assume that the sequence $(x_k)_{k \in \N}$ is generated by \cref{alg: multilevel}. Then, for any $k \in \N$, the following statements hold true:
    \begin{enumerate}
        \item     If $\xi \leq 1$, then
    \begin{equation}
        \| \nabla f_{k+1} \| \leq \frac{L \xi^4}{2 \sigma \mu^2 \omega^4} \|\nabla f_{k} \|^2. \label{ineq: quadratic rate xi <1}
    \end{equation}
    \item  If $\xi > 1$, then
    \begin{equation}
        \| \nabla f_{k+1} \| \leq \frac{L \xi^5}{2 \sigma \mu^2 \omega^4} \|\nabla f_{k} \|^2. \label{ineq: quadratic rate xi >1}
    \end{equation}
    \end{enumerate}
\end{theorem}
\begin{proof}
    Let us fix $k \in \N$. We consider the first case where $\xi \leq 1$. If at iteration $k$ a coarse step \eqref{eq: coarse step} is performed, then
    \begin{equation}
             \sigma \| \nabla f_{k+1}\|  \overset{\eqref{ineq:conditions-coarse-model-new R_k}}{\leq} \|R_{i,k} \nabla f_{k+1} \| \overset{\eqref{ineq: strongly convex rate subspace}}{\leq} \frac{L \xi^5}{2 \mu^2 \omega^4} \|\nabla f_{k} \|^2 \leq \frac{L \xi^4}{2 \mu^2 \omega^4} \|\nabla f_{k} \|^2. \label{ineq: fisrt case k xi <1}
    \end{equation}
    On the other hand, if the Newton step \eqref{eq: newton step} is used and since $0 < \sigma \leq 1$, we obtain
    \begin{equation}
        \sigma \| \nabla f_{k+1}\| \leq \| \nabla f_{k+1} \| \overset{\eqref{ineq: newton's quadratic rate}}{\leq} \frac{L}{2 \mu^2 } \|\nabla f_{k} \|^2 \overset{\eqref{ineq: norm R_i, omega and xi}}{\leq} \frac{L \xi^4}{2 \mu^2 \omega^4} \|\nabla f_{k} \|^2, \label{ineq: second case k xi<1}
    \end{equation}    
    and hence by \eqref{ineq: fisrt case k xi <1} and \eqref{ineq: second case k xi<1}, inequality \eqref{ineq: quadratic rate xi <1} is proved.
    
    Assume now that $\xi > 1$. If at iteration $k$ we use \eqref{eq: coarse step} to compute $x_{k+1}$, then
    \begin{equation}
             \sigma \| \nabla f_{k+1}\|  \overset{\eqref{ineq:conditions-coarse-model-new R_k}}{\leq} \|R_{i,k} \nabla f_{k+1} \| \overset{\eqref{ineq: strongly convex rate subspace}}{\leq} \frac{L \xi^5}{2 \mu^2 \omega^4} \|\nabla f_{k} \|^2. \label{ineq: fisrt case k xi >1}
    \end{equation}
    On the other hand, if the Newton step \eqref{eq: newton step} is performed, then
    \begin{equation}
        \sigma \| \nabla f_{k+1}\| \overset{\eqref{ineq:conditions-coarse-model-new R_k}}{\leq} \| \nabla f_{k+1} \| \overset{\eqref{ineq: newton's quadratic rate}}{\leq} \frac{L}{2 \mu^2 } \|\nabla f_{k} \|^2 \overset{\eqref{ineq: norm R_i, omega and xi}}{\leq} \frac{L \xi^4}{2 \mu^2 \omega^4} \|\nabla f_{k} \|^2 \leq \frac{L \xi^5}{2 \mu^2 \omega^4} \|\nabla f_{k} \|^2, \label{ineq: second case k xi>1}
    \end{equation}    
    and hence by \eqref{ineq: fisrt case k xi >1} and \eqref{ineq: second case k xi>1}, inequality \eqref{ineq: quadratic rate xi >1} is proved.
\end{proof}

Theorem \ref{thm: rate strong convexity} shows the progress of the norm of the gradient after taking one step in the \cref{alg: multilevel}, and that this progress depends on the magnitude of $\xi$. Further, \cref{thm: rate strong convexity} provides us with the following characterization about the local convergence of \cref{alg: multilevel}. Let us consider the case $\xi \leq 1$ (the case $\xi>1$ follows analogously). If  $\|\nabla f_{k} \| < \eta \coloneq \frac{2 \sigma \mu^2 \omega^4}{L \xi^4}$, then there exists a $C \in [0,1)$ such that $\|\nabla f_{k} \| \le C \eta$. Moreover, let $\ell \in \N \setminus \{0\}$. Then, applying \eqref{ineq: quadratic rate xi <1} recursively, we obtain
\begin{equation}\label{eq:convproof2}
\frac{1}{\eta} \| \nabla f_{k+\ell} \| \leq \left(\frac{1}{\eta} \| \nabla f_{k} \| \right)^{2^{\ell + 1}} \leq C^{2^{\ell + 1}},
\end{equation}
% \begin{equation} \label{eq:convproof1}
%             \| \nabla f_{k+1} \| \leq \frac{L \xi^4}{2 \sigma \mu^2 \omega^4} \|\nabla f_{k} \|^2 \le C \| \nabla f_{k} \|.
% \end{equation}
% Furthermore, we have $\| \nabla f_{k+1} \| < \frac{2 \sigma \mu^2 \omega^4}{L \xi^4}$ and thus, by induction,  
% \begin{equation}\label{eq:convproof2}
% \lim_{\ell \to \infty} \| \nabla f_{k+\ell} \| \leq \lim_{\ell \to \infty} C^\ell \|\nabla f_{k} \| = 0,
% \end{equation}
hence $\lim_{\ell \to \infty} \| \nabla f_{k+\ell} \| = 0$ with a quadratic convergence rate. Moreover, \cref{thm: rate strong convexity} shows that once \cref{alg: multilevel} enters its quadratic convergence, it maintains the quadratic rate for all subsequent iterations. This behavior is similar to that of the Newton method. It is expected for the \cref{alg: multilevel} to reach very high accuracy in a few iterations. We postpone a comparison between the quadratic convergence of \cref{alg: multilevel} and Newton's method for \cref{sec: discussion on theory}. In the remark below we attempt to explain intuitively why condition \eqref{ineq:conditions-coarse-model-new R_k} enforces a quadratic rate.

\begin{remark} \label{remark: quadratic rate theorem}
Let us for simplicity drop the $k$ dependence from the operators $R_{i, k}$ and consider \eqref{ineq: upper step before}. Inequality \eqref{ineq: upper step before} shows a quadratic convergence rate the in the subspace spanned by the rows of $R_i$. Then, there exists an iteration $k$ such that for all $1 \leq i \leq m$ \eqref{ineq: upper step before} is satisfied. Since \eqref{ineq:conditions-coarse-model-new R_k} requires the norm $\|\nabla f_{k+1} \|$ to be close to $\|R_i \nabla f_{k+1} \|$ and the latter decays quadratically, then \eqref{ineq:conditions-coarse-model-new R_k} forces the former norm to decrease with similar rate. In particular, according to \cref{thm: rate strong convexity}, if condition \eqref{ineq:conditions-coarse-model-new R_k} selects $1\leq i \leq m-1$, then the quadratic rate is enforced by the quadratic rate of the coarse step in the subspace, otherwise, if $i = m$ (fine level), it is guaranteed by Newton's algorithm.  An important observation is that the quadratic convergence in $\R^n$ is activated if and only if $\|R_i \nabla f_{k} \|$ enters its quadratic phase in the subspace. In general, condition \eqref{ineq:conditions-coarse-model-new R_k} ensures that the norms will decay similarly.
\end{remark}

\begin{remark}
We note that the quadratic convergence rate for \textsc{ML--Newton} in the subspace (see \eqref{ineq: upper step before}) is not new. Similar results appear in \cite{ho2019newton}. Our proof, however, is different and more general, as it supports iteration-dependent operators.
\end{remark}

\subsection{Self-concordant Functions} \label{sec: self conc analysis}
Self-concordance is crucial when analyzing Newton's method as it offers local convergence rates that are invariant to the affine transformation of variables. These rates are important, and are on par with the Newton's method's iterations; it seems that the class of self-concordant functions captures effectively the convergence behavior of Newton's method observed in practical applications. In this section, in a similar fashion, we attempt to explain our algorithm's practical behavior using strictly convex self-concordant functions. In particular, we will show a local quadratic convergence rate that it is independent of unknown constants such as $\mu, L$ involved in \cref{thm: rate strong convexity}.

Since we are dealing with self-concordant functions we will measure the algorithm's convergence using $\lambda(x)$ in \eqref{eq: newton decrement}. For the purposes of our analysis we introduce the quantity
\begin{align}
    g_{i, k}(x) & \coloneq \left[ \left(R_{i, k} \nabla f(x)\right)^\transp \left(R_{i, k} \nabla^2 f(x)P_{i, k}\right)^{-1} \left(R_{i, k} \nabla f(x)\right)\right]^{1/2}. \label{eq: def weird decrement}
\end{align}
The above quantity is well-defined due to $R_{i, k} \nabla^2 f(x)P_{i, k} \succ 0$. When setting $x \coloneq x_k$, it coincides with the approximate decrement $\hat{\lambda}_{i, k}$ used for the analysis of past multilevel algorithms \cite{ho2019newton, tsipinakis2024multilevel, tsipinakis2023lowrank}. It represents the local norm of the reduced gradient induced by the reduced Hessian matrix at iteration $k$ and level $i$. In what follows we collect results for self-concordant functions that will be useful for establishing the convergence results. Let $r \coloneq \|y-x\|_x$, where $\|\cdot\|_x$ is defined in \eqref{eq: local norms definition}. Given \eqref{ineq: def self concordance}, for any $x, y \in \R^n$ such that $Mr < 1$, it is possible to show \cite{nesterov2018lectures} that
\begin{align*}
    (1-Mr)^2 \nabla^2 & f(x)  \preceq \nabla^2 f(y)  \preceq  \frac{1}{(1-Mr)^2}\nabla^2 f(x), \\
    \left(1 - Mr + \frac{1}{3} M^2 r^2\right) \nabla^2 f(x) & \preceq G \coloneq \int_0^1 \nabla^2 f(x + t(y - x)) \d t \preceq \frac{1}{1 - Mr} \nabla^2 f(x).
\end{align*}
From the coarse step \eqref{eq: coarse step}, setting $y \coloneq x_{k+1}$ and $x \coloneq x_k$, we have that $r_k \coloneq \|x_{k+1}-x_k\|_{x_k} \overset{\eqref{eq: coarse step}}{=} \|d_{m,k}\|_{x_k} = \hat{\lambda}_{i, k}$. We will be using a short notation for the reduced gradient, Hessian, and average Hessian, i.e., $\nabla f_{i, k} \coloneq R_{i,k} \nabla f_k$,  $\nabla^2 f_{i, k} \coloneq R_{i,k} \nabla^2 f_k P_{i,k}$, and $G_{i, k} \coloneq R_{i,k} G_k P_{i,k}$, respectively. 

Left- and right-multiplying the above inequalities with $R_{i,k}$ and $P_{i, k}$, respectively, and, if $M \hat{\lambda}_{i, k} < 1$, then,
\begin{align}
    \big(1-M\hat{\lambda}_{i, k}\big)^2  \nabla^2 f_{i, k}   \preceq R_{i,k} \nabla^2 f_{k+1} P_{i,k} & \preceq  \frac{1}{\big(1-M\hat{\lambda}_{i, k}\big)^2} \nabla^2 f_{i,k}, \label{ineq: bounds hessians} \\
    \left(1 - M \hat{\lambda}_{i, k} + \frac{1}{3} M^2 \hat{\lambda}_{i, k}^2\right) \nabla^2 f_{i, k}  \preceq G_{i, k}   & \preceq \frac{1}{1 - M\hat{\lambda}_{i, k}} \nabla^2 f_{i, k}. \label{ineq: bounds G}
\end{align}
Moreover, using the short notation we get
\begin{equation} \label{eq: approx decrement short}
    \hat{\lambda}_{i, k} = \left[\nabla f_{i, k}^\transp \nabla^2 f_{i, k}^{-1} \nabla f_{i, k} \right]^{1/2}.
\end{equation}
The following result, proved in \cite[Lemma 2.3]{tsipinakis2024multilevel}, shows that the approximate decrement is always smaller than the Newton decrement for all $k \in \N$, i.e.,
\begin{equation} \label{ineq: upper bound approx decrement}
    \hat{\lambda}_{i, k} \leq \lambda(x_k).
\end{equation}
For our analysis we will need also the following lemma, proved in \cite{tsipinakis2025convergence}. 
\begin{lemma}[Lemma 3.1 in \cite{tsipinakis2025convergence}] \label{lemma: quadratic form bound on matrices}
Let $H, G$ be symmetric positive definite matrices and $a, b \in \R$ such that
\begin{equation} \label{ineq: ass in lemma appendix}
     a H \preceq G - H \preceq b H.
\end{equation}
Then,
\begin{equation} \label{ineq: quadratic form bound on matrices}
    \left( G - H\right) H^{-1} \left( G - H\right) \preceq c^2 H,
\end{equation}
where $c = \max\{|a|, |b| \}$.
\end{lemma}
The following lemma is key to our analysis. 
\begin{lemma} \label{lemma: self concordant rate subspace}
Let $M \lambda(x_k) < 1$. If the point $x_{k+1}$ is obtained by the coarse model as in \eqref{eq: coarse step}, then
\begin{equation}
    g_{i, k}(x_{k+1}) \leq \frac{M \lambda(x_k)^2}{(1-M\lambda(x_k))^2} \label{ineq: self concor rate subspace}
\end{equation}
for all $i \in \{1,2, \ldots, m-1 \}$ and $k \in \N$.
\end{lemma}

\begin{proof}
    Let us fix $i \in \{1,2, \ldots, m-1 \}$ and $k \in \N$. By our assumption and \eqref{ineq: upper bound approx decrement} we have $M\hat{\lambda}_{i, k} \leq M\lambda(x_k)<1$.
    Note also that
    \begin{equation} \label{eq: identity from taylor short}
        R_{i, k} \nabla f_{k+1} \overset{\eqref{eq: identity from taylor formula}}{=} -R_{i,k} \left( \nabla^2 f_{k} - G_k \right) d_{m,k} 
        \overset{\eqref{eq: coarse direction at m}}{=} - \left( \nabla^2 f_{i, k} - G_{i, k} \right) d_{i, k},
    \end{equation}   
    where 
    \begin{equation} \label{eq: d_i,k short not}
        d_{i, k} = - \nabla^2 f_{i, k}^{-1} \nabla f_{i, k}.
    \end{equation}
    Furthermore, using \eqref{ineq: bounds G} and subtracting $\nabla^2 f_{i, k}$  we get
\begin{align}
     \left(- M \hat{\lambda}_{i, k} + \frac{1}{3} M^2 \hat{\lambda}_{i, k}^2\right)\nabla^2 f_{i, k} & \preceq G_{i, k} - \nabla^2 f_{i, k} \preceq \frac{M\hat{\lambda}_{i, k}}{1 - M\hat{\lambda}_{i, k}} \nabla^2 f_{i, k} \notag \\
    \implies  - \frac{M\hat{\lambda}_{i, k}}{1 - M\hat{\lambda}_{i, k}}\nabla^2 f_{i, k} & \preceq G_{i, k} - \nabla^2 f_{i, k} \preceq \frac{M\hat{\lambda}_{i, k}}{1 - M\hat{\lambda}_{i, k}} \nabla^2 f_{i, k}.
\end{align}
Thus, applying \cref{lemma: quadratic form bound on matrices} with $c = \frac{M\hat{\lambda}_{i, k}}{1 - M\hat{\lambda}_{i, k}}$ we have that
\begin{equation} \label{ineq: quadratic form for matrices}
    \left( G_{i, k} - \nabla^2 f_{i, k} \right) \nabla^2 f_{i, k}^{-1} \left(  G_{i, k} - \nabla^2 f_{i, k} \right) \preceq \left( \frac{M\hat{\lambda}_{i, k}}{1 - M\hat{\lambda}_{i, k}} \right)^2 \nabla^2 f_{i, k}.
\end{equation}
  Hence,
    \begin{align}
            g_{i, k}(x_{k+1}) & \overset{\eqref{eq: def weird decrement}}{=} \left[ \left(R_{i, k} \nabla f_{k+1}\right)^\transp \left(R_{i, k} \nabla^2 f_{k+1} P_{i, k}\right)^{-1} \left(R_{i, k} \nabla f_{k+1} \right)\right]^{1/2} \notag \\
            & \overset{\eqref{ineq: bounds hessians}}{\leq} \frac{1}{1 - M\hat{\lambda}_{i, k}}
            \left[ \left(R_{i, k} \nabla f_{k+1} \right)^\transp \left(R_{i, k} \nabla^2 f_{k} P_{i, k}\right)^{-1} \left(R_{i, k} \nabla f_{k+1} \right)\right]^{1/2} \notag \\
            & \overset{\eqref{eq: identity from taylor short}}{=} \frac{1}{1 - M\hat{\lambda}_{i, k}} \left[d_{i, k}^\transp \left( G_{i, k} - \nabla^2 f_{i, k} \right) \nabla^2 f_{i, k}^{-1} \left(  G_{i, k} - \nabla^2 f_{i, k} \right) d_{i, k}  \right]^{1/2} \notag \\
            & \overset{\eqref{ineq: quadratic form for matrices}}{\leq} \frac{M\hat{\lambda}_{i, k}}{\big(1 - M\hat{\lambda}_{i, k}\big)^2} \left[d_{i, k}^\transp \nabla^2 f_{i, k} d_{i, k}  \right]^{1/2} \notag \\
            & \overset{\eqref{eq: d_i,k short not}}{=} \frac{M\hat{\lambda}_{i, k}^2}{\big(1 - M\hat{\lambda}_{i, k}\big)^2} \notag \\
            & \overset{\eqref{ineq: upper bound approx decrement}}{\leq}  \frac{M\lambda(x_k)^2}{\left(1 - M\lambda(x_k)\right)^2}, \notag \label{ineq: last bound on norm with newton decrement}
    \end{align}
    where the last inequality holds since $x \mapsto \frac{M x^2}{(1-Mx)^2}$ is increasing on $(0, 1/M)$. This completes the proof.
\end{proof}

The result of \cref{lemma: self concordant rate subspace} is analogous to that of \cref{lemma: strongly convex rate subspace} but uses the local norms instead. 
To show the quadratic rate we require to look one step ahead to determine whether the future coarse step is efficient. Our main condition in \eqref{ineq:conditions-coarse-model-new R_k} must then be adjusted accordingly. That is, the coarse model at level $i$ is accepted if
\begin{equation} \label{ineq: conditions self-concordant}
    g_{i, k}(x_{k+1}) \geq \sigma \lambda(x_{k+1}),
\end{equation}
where $\sigma \in (0, 1]$. Note that such conditions require the computation of the inverse Hessian matrix at each iteration and thus they are not efficient for practical applications. However, they are important for enhancing our understanding of  the iterates of \cref{alg: multilevel}. In practice, even when minimizing a self-concordant function, one should check \eqref{ineq:conditions-coarse-model-new R_k}, which is cheaper and serves the same purpose.

\begin{theorem} \label{thm: quadratic rate self conc}
    Assume that the sequence $(x_k)_{k \in \N}$ is generated by \cref{alg: multilevel}. Then, if $M \lambda(x_k) < 1$,
    \begin{equation} \label{ineq: quadratic rate self conc}
        \lambda(x_{k+1}) \leq \frac{M \lambda(x_{k})^2}{\sigma (1-M\lambda(x_{k}))^2}
    \end{equation}
    for any $k \in \N$.
\end{theorem}
\begin{proof}
    Let us fix $k \in \N$. If $x_{k+1}$ is obtained by \eqref{eq: coarse step}, then
    \begin{equation*}
      \sigma \lambda(x_{k+1}) \overset{\eqref{ineq: conditions self-concordant}}{\leq} g_{i, k}(x_{k+1}) \overset{\eqref{ineq: self concor rate subspace}}{\leq} \frac{M \lambda(x_k)^2}{(1-M\lambda(x_k))^2}.
    \end{equation*}
    On the other hand, if $x_{k+1}$ is obtained by \eqref{eq: newton step}, then
    \begin{equation*}
        \sigma \lambda(x_{k+1}) \leq \lambda(x_{k+1}) \overset{\eqref{ineq: newton quadratic phase self concor}}{\leq} \frac{M \lambda(x_{k})^2}{\left( 1-M \lambda(x_{k})\right)^2},
    \end{equation*}
    where the first inequality follows from $\sigma \in (0, 1]$. Thus, the proof is finished.
\end{proof}

Theorem \ref{thm: quadratic rate self conc} show the progress of the local norm of the gradient after taking one step in \cref{alg: multilevel}. It provides us with the following description about the local convergence of the method: If $\lambda(x_{k}) < \frac{2\sigma + 1 - \sqrt{4 \sigma + 1}}{2 M \sigma}$, then 
\begin{equation*}
    \lambda(x_{k+1}) \leq \frac{M \lambda(x_{k})^2}{\sigma (1-M\lambda(x_{k}))^2} < \lambda(x_{k}),
\end{equation*}
and in particular, similar to \eqref{eq:convproof2}, for $\ell \ge k$, we have  $\lim_{\ell \to \infty} \lambda(x_{\ell}) = 0$ with quadratic convergence rate. The algorithm will maintain the quadratic rate for all $k$ once it enters this phase. Note that this result is more informative than that of \cref{thm: rate strong convexity}. For $M=1$ (standard self-concordant functions), it does not depend on unknowns such as Lipschitz and strong convexity parameters. More importantly, compared to \cref{thm: rate strong convexity}, the rate above also remains invariant for any choice of $R_{i,k}$, and thus the theory of  \textsc{AML--Newton} method aligns with that of the classical Newton method.

\subsection{Discussion of the Convergence Results} \label{sec: discussion on theory}
Let us recall the main results of this section. For simplicity, assume that $\xi = \omega = 1$ in \eqref{thm: rate strong convexity}. Then 
\begin{equation} \label{ineq: quadratic rates together}
    \| \nabla f(x_{k+1}) \| \leq \frac{L}{2 \sigma \mu^2 } \|\nabla f(x_{k}) \|^2 \quad \text{and} \quad \lambda(x_{k+1}) \leq \frac{M \lambda(x_{k})^2}{\sigma (1-M\lambda(x_{k}))^2}.
\end{equation}
An important attribute of \cref{alg: multilevel} is that the quadratic rate is governed by the user-defined parameter $\sigma \in (0, 1]$. If the user selects $\sigma$ small to allow for more coarse steps, then the above bounds yield slower rate of convergence in the quadratic phase. However, note that, since the quadratic phase is very fast (convergence attained in a few iterations), the main effort of \cref{alg: multilevel} is spent in the slower first phase. Recall also that, due to \eqref{ineq: quadratic rates together}, the size of the region of quadratic convergence is proportional to $\sigma$. Hence, if $\sigma$ is small then \cref{alg: multilevel} selects smaller coarse models but the phase of quadratic convergence will be activated in late stages of the minimization process. As a result, there is also a trade-off between the choice of $\sigma$ and the size of the region of quadratic convergence. Insights regarding the region of quadratic rate may also be given by the following theorem.

\begin{theorem} \label{thm: rate interpretation}
    Suppose that $\|R_{i, k} \nabla f(x_{k})\| < 2 \mu^2 / L$ for all $i \in \{1,2, \ldots, m \}$. Suppose also that \eqref{ineq:conditions-coarse-model-new R_k} holds and define 
    \begin{equation*}
            a_{i, k} \coloneq \frac{\|R_{i, k} \nabla f(x_{k+1})\|}{\|R_{i, k} \nabla f(x_{k})\|}.
    \end{equation*}
    Then there exists $K_\sigma \in \N$ such that for all $k \ge K_\sigma$ we have
    \begin{equation*}
         \max_i \{ a_{i, k} \} < \sigma \quad \text{and} \quad \|\nabla f(x_{k+1})\| < \|\nabla f(x_{k})\|.
    \end{equation*}
\end{theorem}

\begin{proof}
    Let us fix $i \in \{1,2, \ldots, m \}$ and $k \in \N$, and recall that we consider $\omega = \xi = 1$. We have
    \begin{equation*}
         a_{i, k} \overset{\eqref{ineq: upper step before}}{\leq} \frac{L}{2 \mu^2 } \| R_{i, k} \nabla f(x_{k}) \|.
    \end{equation*}
    Moreover, by \eqref{ineq: upper step before} and $\|R_{i, k} \nabla f(x_{k})\| < 2 \mu^2 / L$,  the sequence $\| R_{i, k} \nabla f(x_{k}) \|$ converges to zero (similarly as to the explanation after the proof of Theorem~\ref{thm: rate strong convexity}), and thus $(a_{i, k})_{k \in \N}$ is a null sequence. Then there exists $K_\sigma \in \N$ such that
    $a_{i, k} < \sigma$ for all $k \ge K_\sigma$ and all $i \in \{1,2, \ldots, m \}$, which proves the left inequality of the theorem. In addition,
    \begin{equation*}
         \| \nabla f(x_{k+1})\|  \overset{\eqref{ineq:conditions-coarse-model-new R_k}}{\le} \frac{1}{\sigma} \|R_{i, k} \nabla f(x_{k+1}) \| = \frac{a_{i, k}}{\sigma}  \|R_{i, k} \nabla f(x_{k}) \|  <  \|\nabla f(x_{k}) \|,
    \end{equation*}
    and thus the proof is finished.
\end{proof}
% derived through the conditions \eqref{ineq:conditions-coarse-model-new R_k} and \eqref{ineq: conditions self-concordant}. Let us consider the strongly convex case. Let us also fix some $1 \leq i \leq m$ and $k \in \N$, where for $i=m$ we set $R_{m, k} \coloneq I_n$. Denote
% \begin{equation*}
%     a_{i, k} \coloneq \frac{\|R_{i, k} \nabla f(x_{k+1})\|}{\|R_{i, k} \nabla f(x_{k})\|},
% \end{equation*}
% and suppose that the quadratic rate in the subspace is activated. Then, by \eqref{ineq: upper step before}, we have $a_{i, k} \leq \frac{L}{2 \mu^2 } \| R_{i, k} \nabla f(x_{k}) \|$.  For some $0 < \sigma \leq 1$ we have
% \begin{equation*}
%  \sigma \| \nabla f(x_{k+1})\|  \leq \|R_{i, k} \nabla f(x_{k+1}) \| = a_{i, k}  \|R_{i, k} \nabla f(x_{k}) \| \leq  a_{i, k} \|\nabla f(x_{k}) \|.
% \end{equation*}
\cref{thm: rate interpretation} shows that the quadratic rate of \cref{alg: multilevel} will be activated if  $\sigma > a_{i, k} $ for all  $i \in \{1,2, \ldots, m \}$. 
Thus, even though the quadratic rate in the subspace is activated already, if $\sigma$ is chosen small,  the quadratic rate in $\R^n$ will be activated only in the late stages of the quadratic rate in $\R^{n_i}$. On the other hand, if $\sigma = 1$, the local region of quadratic rates in $\R^{n}$ and $\R^{n_i}$ coincide (which is also obvious from \eqref{ineq: upper step before} and \eqref{ineq: quadratic rates together}). 

When it is time for \cref{alg: multilevel} to enter the quadratic phase, it has to select coarse models in $\R^{n_i}$ such that the decrease in $\| \nabla f(x_{k})\|$ is similar to that observed by the trial step for $\| R_i \nabla f(x_{k})\|$. However, note that the latter norm decreases quadratically and with rate that matches that of the Newton method. Hence, \cref{alg: multilevel} attempts to match the convergence rate of Newton's method, given some $\sigma \in (0,1]$. It is natural then for \cref{alg: multilevel}, in order to satisfy this requirement, to choose subspaces where $n_i$ is large enough or even $n_i=n$.  Then, one should expect the \cref{alg: multilevel} to take expensive steps, that resemble the complexity of Newton's method, once the quadratic phase is activated.

We will finish this section by comparing \cref{alg: multilevel} and Newton's method. Note that the only difference between the estimates \eqref{ineq: quadratic rates together} and Newton's estimates in  \eqref{ineq: newton's quadratic rate} and \eqref{ineq: newton quadratic phase self concor} is the presence of $0 < \sigma \leq 1$ in \eqref{ineq: quadratic rates together}. Thus, Newton's local quadratic rate is faster than that of \cref{alg: multilevel}. Since, however, quadratic rates are very fast, the main difference between \cref{alg: multilevel} and Newton's method comes from the region of their local convergence. The estimates \eqref{ineq: quadratic rates together} imply a restricted region for the local quadratic convergence, and thus \cref{alg: multilevel} requires more steps in the first phase than Newton's method. However, this is not necessarily a disadvantage of \cref{alg: multilevel} because it generates fast iterates in levels with $n_i < n$ or even $n_i \ll n$.  On the other hand, it is known that Newton's global convergence achieves a linear rate \cite{karimireddy2018global}, and thus it does not justify its high computational demands. Therefore, one should find a replacement of Newton's method in the first phase and only apply Newton in the second phase. This is exactly what \cref{alg: multilevel} does, and importantly, through condition \eqref{ineq:conditions-coarse-model-new R_k}, it offers an automatic way to switch from linear to a Newton-like quadratic convergence in the late stages of optimization. We will verify the theoretical findings of this section in \cref{sec: experiments}.

\subsection{Relation to Randomized Subspace Newton Method} Newton-type subspace methods have attracted a lot of interest due their efficiency for large-scale optimization \cite{gower2019rsn, hanzely2020stochastic, tsipinakis2025regularized}. By keeping the subspace fixed, the key difference to multilevel methods is that they can converge to the minimizer of $f$ without ever computing directions on the fine level. This fact makes subspace methods superior to multilevel methods in practical applications where $n$ is prohibitively large to form the Hessian matrix. Regarding their convergence rate, they achieve  probabilistic linear convergence rates, unless extra assumptions related to the problem structure are imposed. 

Due to the generic framework developed in sections \ref{sec: new multilevel} and \ref{sec: analysis}, \cref{alg: multilevel} can be directly seen as the RSN. Let us % fix some $1 \leq i \leq m-1$ and
define surjective linear mappings $R_k : \R^n \to \R^{n_{\rm low}}$ for $k \in \N$, where $n_{\rm low} < n$. Then, the corresponding restriction operator at some iteration $k$ is denoted by $R_k$ and we set $P_k \coloneq R_k^\transp$. Then the iterates for the subspace Newton method are given as
\begin{equation} \label{eq: subspace methods iterates}
    x_{k+1} = x_k - P_k \left(R_k \nabla^2 f(x_k) P_k \right)^{-1} R_k \nabla f(x_k).
\end{equation}
RSN method generates $x_{k+1}$ randomly at each iteration by drawing the entries of $R_k$ from a probability distribution $\mathcal{D}$. We can use the basic Johnson-Lindenstrauss (JL) result to show a local probabilistic quadratic convergence rate for the scheme \eqref{eq: subspace methods iterates}. 
\begin{lemma}[\cite{johnson1984extensions, kane2014sparser}] \label{lemma: JL probabilistic}
    For any $\varepsilon \in (0, 1/2)$ and $\delta \in (0, 1/2)$ there exist a scalar $C>0$ and a probability distribution $\mathcal{D}$ on the set of $n_{\rm low} \times n$ real matrices with $n_{\rm low} \coloneq C \varepsilon^{-2} \log(1/\delta)$ such that for any $x \in \R^n$ and $k \in \N$,
    \begin{equation} \label{ineq: JL probabilistic}
        \mathbb{P}_{R_k \sim \mathcal{D}} \left( (1-\varepsilon)\| x \| \leq \|R_k x\| \leq  (1+\varepsilon) \| x \| \right) > 1 - \delta.
    \end{equation}
\end{lemma}

\begin{theorem} \label{cor: probabilistic rate strong convexity}
    Let $\varepsilon \in (0, 1/2)$, $\delta \in (0, 1/2)$, $C$ and $\mathcal{D}$ be as in \cref{lemma: JL probabilistic}.
    Suppose that the entries of $R_k \in \R^{n_{\rm low} \times n}$ for $k \in \N$ are drawn from $\mathcal{D}$ and that the sequence $(x_k)_{k \in \N}$ is generated by \cref{eq: subspace methods iterates}. If $n_{\rm low} \geq C \varepsilon^{-2} \log(1/\delta)$, then for any $k \in \N$,
    \begin{equation} \label{ineq: quadratic rate probablistic}
      \mathbb{P}_{R_k \sim \mathcal{D}} \left(  \| \nabla f_{k+1} \| \leq \frac{L}{2 \mu^2 }\left(\frac{1+\varepsilon}{1 - \varepsilon} \right)^5 \|\nabla f_{k} \|^2 \right) > (1-\delta)^2.
    \end{equation}
\end{theorem}

\begin{proof}
    From \cref{ineq: JL probabilistic}, with probability at least $1-\delta$, we immediately obtain $\omega \ge 1-\varepsilon$ and $\xi \le 1 + \varepsilon$ in \eqref{eq: def omega and xi}. 
    Thus, inequalities \eqref{ineq: bounds on RP} and \eqref{ineq: upper bound on d_m} become
        \begin{equation*}
        R_{i, k} P_{i, k} \succeq \sigma_{n_i}^2 I_n \succeq (1-\varepsilon)^2 I_n \quad \text{and} \quad \| d_{m, k}\| \leq \frac{1 + \varepsilon}{\mu (1-\varepsilon)^2} \| R_{i, k}\nabla f_k \|,
    \end{equation*}
    where both remain true with probability at least $1-\delta$. Hence, the result of \cref{lemma: strongly convex rate subspace} becomes 
    \begin{equation*}
        \|R_{i,k} \nabla f_{k+1} \| \leq \frac{L (1 + \varepsilon)^5}{2 \mu^2 (1-\varepsilon)^4} \|\nabla f_{k} \|^2,
    \end{equation*}
        with probability at least $1-\delta$. Moreover, \eqref{ineq: JL probabilistic} implies that condition \eqref{ineq:conditions-coarse-model-new R_k} holds with probability at least $1-\delta$ and  $\sigma = 1-\varepsilon$. Using this and applying the proof of \cref{thm: rate strong convexity}, inequality \eqref{ineq: quadratic rate probablistic} follows immediately.
\end{proof}

\cref{cor: probabilistic rate strong convexity} shows that for sufficiently large $n_{\rm low}$, if $\|\nabla f_{k} \| < \frac{2 \mu^2}{L }\left(\frac{1-\varepsilon}{1 + \varepsilon} \right)^5$, then 
\begin{equation*}
     \| \nabla f_{k+1} \| \leq  \frac{L}{2 \mu^2 } \left(\frac{1+\varepsilon}{1 - \varepsilon} \right)^5 \|\nabla f_{k} \|^2 < \|\nabla f_{k} \|
\end{equation*}
holds with probability at least $(1-\delta)^2$ for each $k$. That is, at each iteration, the subspace method advances to $k+1$ quadratically with high probability. More importantly, it implies high total probability of converging quadratically. In particular, suppose that the algorithm takes $K_\epsilon$ steps in the the second phase. Then, the total probability to converge quadratically is $(1-\delta)^{2K_{\epsilon}}$. Since quadratic convergence is very fast ($K_{\epsilon}$ is at most $5$ or $6$ in practice) and $\delta$ is small by \cref{cor: probabilistic rate strong convexity}, then the total probability $(1-\delta)^{2K_{\epsilon}}$ remains high. Particular examples of $\mathcal{D}$ that satisfy \cref{lemma: JL probabilistic} include but are not limited to Gaussian distributions, e.g., Gaussian matrices, and uniform distributions, e.g., 1- and s-hashing matrices, for more details see \cite{cartis2022randomised} and references therein. Explicit convergence rates with these particular examples can be derived in a similar manner, nevertheless, the theorem does not seem practical since quadratic rates in the subspace could not be observed (which is a requirement in our theory, see \cref{remark: quadratic rate theorem}). However, \cref{cor: probabilistic rate strong convexity} is important as it demonstrates that probabilistic quadratic rates are possible. Specifically, it suggests that, for such rates to be attainable, one should design restriction operators $R_k$ that satisfy \cref{lemma: JL probabilistic} and ensure that the quadratic rate is activated within the relevant subspace. A probabilistic analysis for the class self-concordant functions is not straightforward as it requires a similar result to \cref{lemma: JL probabilistic} based on local norms.

\section{Numerical Experiments} \label{sec: experiments}

In this section, we aim to validate our theoretical findings using real-world problems. Specifically, we seek to demonstrate:
\begin{enumerate}
\item the quadratic convergence rate of \cref{alg: multilevel} and the influence of the parameter $\sigma$
(see \eqref{ineq: quadratic rates together}),
\item the extent to which optimization progress can be made without ever taking fine-level steps,
\item that \cref{alg: multilevel} is particularly efficient for problems with low-rank structure,
\item that \cref{alg: multilevel} outperforms Newton's method, Gradient Descent, and the classical multilevel Newton method on real-world problems.
\item that \cref{alg: multilevel} can attain \emph{quadratic convergence} with a \emph{total complexity} lower than that of the RSN \cite{gower2019rsn}.
\end{enumerate}

To address the above questions and assess the  performance of \cref{alg: multilevel}, we consider solving the regularized GLM problem
\begin{equation*}
    \min_{x \in \R^n} f(x) + \frac{l_2}{2} \|x\|^2 +  l_1 h(x), 
\end{equation*}
where $l_1, l_2 >0$, and $ h(x) \coloneq \sum_{i=1}^n \left( \sqrt{c^2 + x_i^2} - c \right)$ is a differentiable approximation of the ${\| \cdot \|}_1$-norm that imposes sparsity to the solution $x^*$ when $c > 0$ is small. Given a collection $\{(a_i, b_i )\}_{i = 1}^d$, where $a_i \in \R^n$ and $b_i \in \R$, we consider two candidates for the function  $f$:
\begin{enumerate}
    \item the \emph{Logistic regression loss function} with $b_i \in \{-1, 1\}$ for $i = 1,\ldots,d$, i.e.,
    \begin{equation*}
        f(x) = \sum_{i=1}^d \left[ \log\left(1 + \e^{a_i^\transp x}\right) - b_i \, a_i^\transp x \right].
    \end{equation*}
    \item the \emph{Poisson regression loss function} with  $b_i \in \N$ for $i = 1,\ldots,d$, i.e.,
    \begin{equation*}
        f(x) = \sum_{i=1}^d \left( \e^{a_i^\transp x} - b_i a_i^\transp x \right).
    \end{equation*}
\end{enumerate}
Both problems are widely used in practical machine learning applications. The logistic loss satisfies the strongly convex assumption  under mild conditions \cite{sypherd2020alpha}, whereas the Poisson loss is a self-concordant function. \cref{tab:problem_datasets} shows details of the datasets used and the associated loss function.

\begin{table}[bt]
\centering
\caption{Overview of experiments and dataset dimensions. The datasets are available from \url{https://www.csie.ntu.edu.tw/~cjlin/libsvmtools/datasets/}, \url{https://www.10xgenomics.com/datasets/pbmc-3-k-1-standard-3-0-0}, \url{https://archive.ics.uci.edu/dataset/275/bike+sharing+dataset} and \url{https://opendata.cityofnewyork.us/}. The Generated dataset is drawn from $\mathcal{N}(0, 1)$ and has rank 10.}
\label{tab:problem_datasets}
\begin{tabular}{l l c c}
\toprule
\textbf{Problem} & \textbf{Dataset} & $d$ & $n$ \\
\midrule
Poisson & Single‐cell RNA   & 2,700   & 10,754 \\
Poisson & Bike      & 17,379   & 12 \\
Poisson & Data 311 & 6,359 & 157\\
Poisson & Generated & 1,000   & 900 \\
Logistic& Ctslices  & 53,500   & 385 \\
Logistic& Duke  & 38   & 7,129 \\
Logistic& Gisette  & 6,000   & 5,000 \\
Logistic& w8a  & 49,749   & 300 \\
Logistic& a9a  & 32,561   & 123 \\
\bottomrule
\end{tabular}
\end{table}

\subsection{Implementation Details} \label{sec: implementation details}
Our ultimate goal is to develop an algorithm that converges rapidly to the optimal solution $x^*$. To this end, we require coarse models that are both informative and of sufficiently small dimension. We begin by considering the case in which \cref{alg: multilevel} constructs coarse models using fixed operators $R_i$. In this setting, the design of the coarse levels plays a crucial role in the algorithm’s success. A poorly chosen design may cause \cref{alg: multilevel} to perform numerous trial coarse steps, significantly increasing the computational complexity. To prevent such inefficiency, the union of the subspaces must span the full space of the original problem, thereby ensuring that the update \eqref{eq: coarse step} can affect all components of $x_k$ for $1 \leq i \leq m-1$. Otherwise, it is natural for \cref{alg: multilevel} to select the Newton step to compensate for missing information in the subspaces. In what follows, we describe a simple procedure for constructing such coarse levels.

Let $S_n \coloneqq \{1, 2, \ldots, n\}$. We construct a family of subsets $S_{n_1}, S_{n_2}, \ldots, S_{n_{m-1}} \subset S_n$, where $n_1 < n_2 < \cdots < n_{m-1} < n$. The construction proceeds in two phases: \medskip
\begin{itemize}
    \item \textbf{Coverage phase (subsets $S_{n_1}, \ldots, S_{n_r}$):}  
    For some $r < m-1$, sample elements uniformly without replacement from $S_n$ to construct subsets of cardinality $n_i$ such that
\begin{equation*}
        \bigcup_{i=1}^r S_{n_i} = S_n.
\end{equation*}
    \item \textbf{Overlap phase (subsets $S_{n_{r+1}}, \ldots, S_{n_m}$):}  
    For each $i > r$, sample $n_i$ elements from $S_n$ according to a probability distribution $p_i$, without replacement, satisfying
\begin{equation*}
        p_i(j) \propto \frac{1}{1 + c_j},
\end{equation*}
    where $c_j$ denotes the number of times the element $j \in S_n$ has already been selected in previous subsets.
\end{itemize} \medskip

The above process ensures that all elements of $S_n$ are included in the subsets and, in addition, it keeps the redundancy balanced. We are now in a position to give the following definition for the restriction  operators $R_i$.

\begin{definition} \label{def: R_i}
Let $S_n \coloneqq \{1, 2, \ldots, n\}$ and define subsets $S_{n_1}, S_{n_2}, \ldots, S_{n_{m-1}} \subset S_n$, where $n_1 < n_2 < \cdots < n_{m-1} < n$.
Further, denote by $s_{n_{i, j}}$ the $j$-th elements of $S_{n_i}$, where $1 \leq j \leq n_i$. We construct $R_i$ as follows: the $j$-th row of $R_i$ is the $s_{n_{i, j}}$-th row of the identity matrix $I_n$, and, in addition, we set $P_i \coloneq R_i^\transp$.
\end{definition}
Note that \cref{def: R_i} satisfies the full-rank assumption on $R_i$ since each element in $S_{n_i}$ is unique. We can similarly define iteration-dependent operators $R_{i,k}$. However, due to their dependence on the iteration index $k$, we can adopt a much simpler construction for the subsets $S_{n_i}$. Specifically, at each iteration, the subsets $S_{n_i}$ are formed by uniform sampling without replacement from $S_n$, thereby ensuring that all elements of $S_n$ are included infinitely many times in the limit. Formally, we have: \medskip
\begin{itemize}
    \item \textbf{Uniformly selected subsets:}  We form each subset $S_{n_i}$, $1 \leq i \leq m-1$, by selecting uniformly without replacement $n_i$ elements from the set $S_n$.
\end{itemize} \medskip
Note that, under \cref{def: R_i}, we have $\omega = \xi = 1$ and thus the convergence rate for strongly convex functions is simplified as in \eqref{ineq: quadratic rates together}.
From now onward, unless otherwise specified, when \cref{alg: multilevel} uses fixed operators, the former definition of the $S_{n_i}$ subsets will be used to construct $R_i$ in \cref{def: R_i}. We call this version of \cref{alg: multilevel} \textsc{AML--Newton} (fixed) as it uses fixed subspaces. If \cref{alg: multilevel} performs steps using iteration-dependent $R_{i, k}$, then they are generated based on the latter definition of the $S_{n_i}$ subsets. We call this version of \cref{alg: multilevel} \textsc{AML--Newton} (random) since the subspaces are generated randomly at each iteration. Both versions of \cref{alg: multilevel} calculate the damping parameter $\alpha_k$ using the Armijo rule.

It is well known that multilevel (or subspace) methods are particularly effective for problems that have a low-rank structure, e.g., their Hessian is nearly low-rank, or the second-order information is concentrated in the leading eigenvalues \cite{ho2019newton, tsipinakis2024multilevel}. In contrast, when the curvature is spread across many directions, the first phase of multilevel methods can be slow, which may prevent the algorithm from entering the quadratic phase. In this case, larger values of $\sigma$ should be selected in order allow for more Newton steps. Similarly, one could randomly permute the $m$ levels at each iteration, introducing an approximately $1/m$ probability of selecting the Newton step. If the problem exhibits a low-rank structure, we wish the coarse models to be employed across the first phase, and hence $\sigma$ should be selected small, nevertheless not too small to compromise the quadratic second stage of the algorithm.

\subsection{Quadratic Rates in Practice} \label{sec: quadratic rates in practice}

In this section, we empirically validate our theoretical results. In particular, we show that, under the setup described in \cref{sec: implementation details}, \textsc{A\textsc{ML--Newton}} exhibits a locally quadratic convergence rate. Moreover, we demonstrate that the size of the region of the quadratic convergence depends on the choice of $\sigma$, as discussed in \cref{ineq: quadratic rates together}. Finally, we illustrate that it is possible to enter the quadratic convergence phase using only coarse steps, however, achieving high-precision solutions ultimately requires a final step computed via Newton’s method.

We design a specific experiment to illustrate our theoretical findings. We set $n_1 = \left\lceil 0.1n \right\rceil$ and $n_{m-1} = n - 1$, and generate the maximum number of levels possible within this range. That is, the hierarchy consists of $n - n_1$ coarse levels, in addition to the original level with $n$ dimensions. The subspaces for both the fixed and random versions of \textsc{AML--Newton} are then constructed as described in \cref{sec: implementation details}, and the corresponding restriction operators are obtained using \cref{def: R_i}. The results appear in \cref{fig:sigma_dataset_comparison} and \cref{tab:sigma_all_datasets}.

\begin{figure}[bt]
    \centering

    % Row 1: CT Slices
    \begin{subfigure}[b]{0.49\textwidth}
        \includegraphics{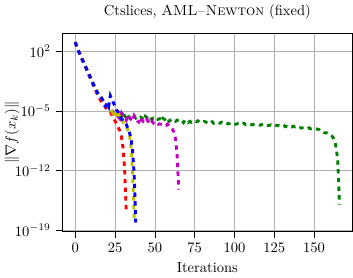}
        \caption{Gradient norm vs. iterations of \textsc{AML--Newton} (fixed) for the  Ctslices dataset}
    \end{subfigure}
    \hfill
    \begin{subfigure}[b]{0.49\textwidth}
    \includegraphics{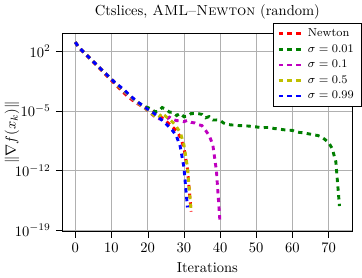}
        \caption{Gradient norm vs. iterations of \textsc{AML--Newton} (random) for the Ctslices dataset}
    \end{subfigure}

    % Row 2: Bike
    \begin{subfigure}[b]{0.49\textwidth}
        \includegraphics{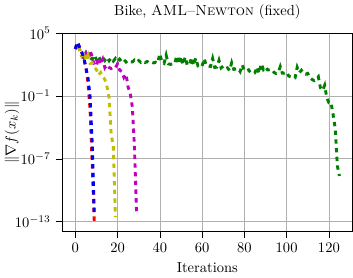}
        \caption{Gradient norm vs. iterations of \textsc{AML--Newton} (fixed) for the Bike dataset}
    \end{subfigure}
    \hfill
    \begin{subfigure}[b]{0.49\textwidth}
         \includegraphics{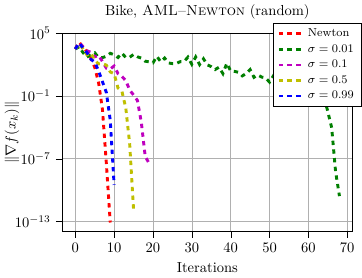}
         \caption{Gradient norm vs. iterations of \textsc{AML--Newton} (random) for the Bike dataset}
    \end{subfigure}

    % Row 3: Example
    % \textbf{Example} \par\vspace{0.3em}
    \begin{subfigure}[b]{0.49\textwidth}
        \includegraphics{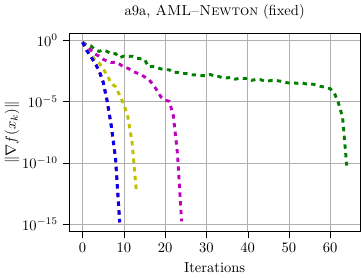}
        \caption{Gradient norm vs. iterations of \textsc{AML--Newton} (fixed) for the a9a dataset}
    \end{subfigure}
    \hfill
    \begin{subfigure}[b]{0.49\textwidth}
      \includegraphics{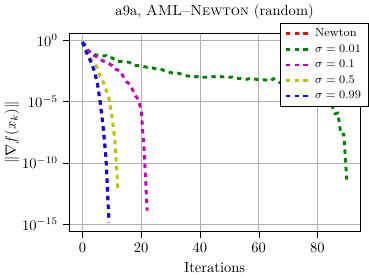}
        \caption{Gradient norm vs. iterations of \textsc{AML--Newton} (random) for the a9a dataset}
    \end{subfigure}

    \caption{The plots illustrate the local quadratic rates of \cref{alg: multilevel} for different choices of $\sigma$ in comparison to those of the Newton method. The number of fine steps taken by \cref{alg: multilevel} for each experiment can be found in \cref{tab:sigma_all_datasets}.}
    \label{fig:sigma_dataset_comparison}
\end{figure}

\begin{table}[bt]
\centering
\caption{Number of fine steps for different $\sigma$ values and three datasets. The results correspond to \cref{fig:sigma_dataset_comparison}.}
\label{tab:sigma_all_datasets}
\begin{tabular}{llcccc}
\toprule
\textbf{Dataset} & \textbf{Algorithm version}  & \multicolumn{4}{c}{\# fine steps for different $\sigma$ values} \\
\cmidrule(lr){3-6}
& & $\sigma = 0.01$ & $0.1$ & $0.5$ & $0.99$ \\
\midrule
\multirow{2}{*}{Ctslices} 
    & \textsc{AML--Newton} (fixed)   & 1 & 1 & 2 & 2 \\
    & \textsc{AML--Newton} (random)   & 1 & 1 & 2 & 2 \\
\midrule
\multirow{2}{*}{a9a} 
    & \textsc{AML--Newton} (fixed)   & 1 & 2 & 2 & 2 \\
    & \textsc{AML--Newton} (random)   & 0 & 0 & 1 & 3 \\
\midrule
\multirow{2}{*}{Bike} 
    & \textsc{AML--Newton} (fixed)   & 1  & 2  & 2 & 3 \\
    & \textsc{AML--Newton} (random)   & 1  & 1  & 2 & 2 \\
\bottomrule
\end{tabular}
\end{table}

In particular, \cref{fig:sigma_dataset_comparison} shows that both versions of \textsc{A\textsc{ML--Newton}} exhibit quadratic convergence, thus verifying the results of \cref{sec: analysis}. Further, the figure illustrates the influence of the parameter $\sigma$ on the size of the region of quadratic convergence. Specifically, when $\sigma$ is chosen small, a more restricted region of quadratic convergence should be expected (see \cref{thm: rate interpretation}), resulting in more steps required by the algorithm in the first phase. On the other hand, when $\sigma$ is close to one, the convergence behavior of \textsc{A\textsc{ML--Newton}} closely matches that of Newton’s method. 
% Thus, the parameter $\sigma$ provides a mechanism for the user to trade off between computational cost and the size of the region of quadratic convergence: a smaller $\sigma$ leads to cheaper steps but it may enter slowly into the quadratic phase, whereas a larger $\sigma$ results in more expensive steps but faster convergence. 
This behavior aligns with the estimates provided in~\eqref{ineq: quadratic rates together}. In all cases, \textsc{A\textsc{ML--Newton}} yields a process that automatically transitions from an inexpensive method in the initial phase to a Newton-like method once the iterates enter the quadratic convergence regime.

Furthermore, as shown in \cref{tab:sigma_all_datasets}, \textsc{A\textsc{ML--Newton}} consistently requires at least one Newton step to attain very high accuracy. In particular, we observe that the algorithm is capable of entering the quadratic convergence phase using only coarse models. However, the final iterate that achieves high-precision accuracy is always computed via a Newton step. This suggests that it is possible to converge quadratically to a very accurate solution without invoking Newton's method. Of course, the coarse models selected in the quadratic phase tend to be of higher dimensionality and resemble the Newton step. 

Moreover, we note that in some instances, \cref{alg: multilevel} appears to enter the quadratic phase only to exit it a few iterations later or not to converge to the same level of accuracy as Newton's method. This behavior disagrees with our theory. One possible explanation is that the method had not entered the quadratic regime in all subspaces simultaneously (see \cref{thm: rate interpretation}). Indeed, in these cases, no Newton step was taken, as indicated in \cref{tab:sigma_all_datasets}. Another common cause is numerical instability at very high precision: even if coarse models reach the quadratic phase, the method may appear to exit as the gradients become tiny.
Because \(\|R_{i,k}\nabla f(x_k)\|\le \|\nabla f(x_k)\|\) (see \cref{def: R_i}) and the quadratic region for \cref{alg: multilevel} is governed by the reduced gradient, an excessively small \(\|R_{i,k}\nabla f(x_k)\|\) can induce ill-conditioning and make \textsc{AML--Newton} stall, falsely suggesting it left the quadratic regime.
We also note that this experiment, in order to confirm and explain the theoretical results, uses an unusually large number of coarse levels, which makes the above instabilities more likely. In practical applications, however, one should only use a few coarse levels, e.g., five in number. In this case, we have not encountered such issues (see the next section), thus, these instabilities are expected to be uncommon in practice.

\begin{figure}[t]
    \centering
    % Row 1: CT Slices
    \begin{subfigure}[b]{0.49\linewidth}
     \includegraphics{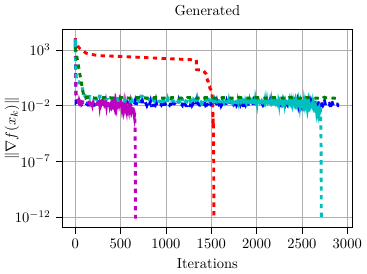}
        \caption{Gradient norm vs. iterations for the Generated dataset with $\ell_1$-regularization}
    \end{subfigure}
    \hfill
    \begin{subfigure}[b]{0.49\linewidth}
      \includegraphics{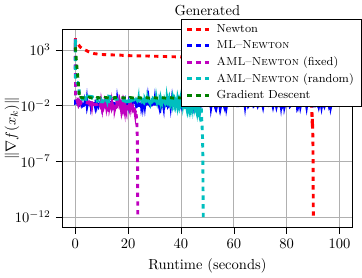}
        \caption{Gradient norm vs. runtime for the Generated dataset with $\ell_1$-regularization}
    \end{subfigure}

    %\vspace{0.8cm}

    % Row 2: Bike
    \begin{subfigure}[b]{0.49\textwidth}
      \includegraphics{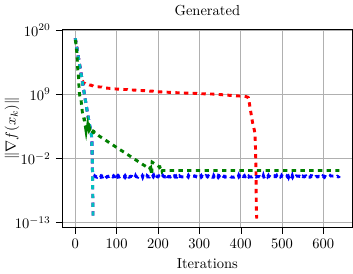}
        \caption{Gradient norm vs. iterations for the Generated dataset with $\mathcal{N}(0, 1)$ initialization}
    \end{subfigure}
    \hfill
    \begin{subfigure}[b]{0.49\textwidth}
       \includegraphics{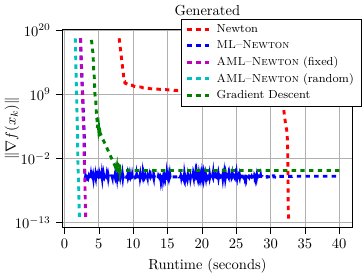}
        \caption{Gradient norm vs. runtime for the Generated datset with $\mathcal{N}(0, 1)$ initialization}
    \end{subfigure}
    \caption{Comparisons between optimization algorithms on Poisson regression with Generated data.
    The left column shows the convergence rate, the right column shows the runtime in seconds. The plot shows the convergence behavior of optimization algorithms for low-rank structured problems. The $\ell_1$-regularization or a bad initial point pose difficulties to optimization methods. Newton's first phase can be extremely slow in such structures, while Gradient Descent and the \textsc{ML--Newton} method converge to sub-optimal solutions.}
    \label{fig:generated_dataset_comparison}
\end{figure}

\subsection{Empirical Comparison of Optimization Algorithms}
In this section we demonstrate the efficiency of \textsc{A\textsc{ML--Newton}} for real-world problems. To obtain efficient multilevel algorithms, we keep the number of levels and coarse model dimensions small. In particular, we make use of $m=6$ models in total (including the fine model), where the coarse models dimensions range from $n_1 = \left\lceil 0.1n \right\rceil$ to $n_5 = \left\lceil 0.3n \right\rceil$, and the intermediate level dimensions are created equidistantly. Since the coarse model dimensions are small, \textsc{A\textsc{ML--Newton}} is expected to enter its quadratic phase selecting mainly the Newton step. The setup of the problems solved in this section is presented in \cref{tab:combined_experiment}. 
\begin{table}[bt]
\centering
\caption{Experimental setups for the Logistic and Poisson regression problems.}
\label{tab:combined_experiment}
\begin{tabular}{llccccl}
\toprule
\textbf{Problem} & \textbf{Dataset} & $x_0$ & $\sigma$ & $l_1$ & $l_2$ \\
\midrule
Logistic & Ctslices  & 0               & 0.2 & $10^{-3}$ & $10^{-6}$ \\
Logistic & Duke      & 0               & 0.1 & 0         & $10^{-6}$ \\
Logistic & Gisette   & 0               & 0.5 & 0         & $10^{-6}$ \\
Logistic & w8        & $\mathcal{N}(0, 1)$ & 0.2 & 0         & $10^{-6}$ \\
Poisson & Generated  & 0               & 0.1 & $10^{-3}$ & $10^{-6}$ \\
Poisson & Generated  & $\mathcal{N}(0, 1)$ & 0.1 & $10^{-3}$ & $10^{-6}$ \\
Poisson & Data 311        & 0 & 0.5 & $10^{-2}$  & $10^{-6}$ \\
Poisson & Single-cell RNA        & 0                  & 0.5 & 0         & $10^{-6}$ \\
\bottomrule
\end{tabular}
\end{table}
As with \eqref{ineq:conditions-coarse-model-new R_k}, the following condition will also be used to check whether the coarse model is selected at iteration $k$, i.e.,
\begin{equation} \label{ineq: extra condition for experiments}
    \|R_{i,k} \nabla f(x_{k})\| \geq \sigma \| \nabla f(x_{k})\|.
\end{equation}
Condition \eqref{ineq: extra condition for experiments} arises from the classical multilevel methods (see \eqref{ineq:conditions-coarse-model-classical}), and it appears to speed-up the algorithm's performance in practice as it discards ineffective coarse models without taking a trial step. Further, for \textsc{AML--Newton} (fixed) and \textsc{ML--Newton}, we noticed that it is efficient to assign a probability distribution $p_i$ for selecting coarse models, since, using the ordinary hierarchy, the conditions tend to select the same levels during the optimization. To avoid this behavior, we randomly permute the 5 coarse levels with each permutation having the same probability. With this, each level $n_i$ has probability $0.2$ to appear first in the hierarchy.  Moreover, since all algorithms included in the comparison have well-established global convergence theories, we employ an Armijo line search to compute the damping parameter at each iteration. The results of the algorithm comparisons are shown in Figures \ref{fig:generated_dataset_comparison}, \ref{fig:logistic regression comparison}, \ref{fig:logistic+poisson regression comparison}, and \ref{fig:generated subspace comparisons}. Moreover, Tables \ref{tab:logistic regression comparison}, \ref{tab:logistic+poisson regression comparison}, and \ref{tab:generated-steps} are associated with Figures \ref{fig:logistic regression comparison}, \ref{fig:logistic+poisson regression comparison}, and \ref{fig:generated subspace comparisons} respectively, and show the total number of coarse and fine steps taken during optimization. 

\begin{table}[bt]
\centering
\caption{Total number of coarse and fine steps computed by multilevel methods during optimization. The results correspond to \cref{fig:logistic regression comparison}.}
\label{tab:logistic regression comparison}
\begin{adjustbox}{max width=\linewidth}
\begin{tabular}{lcccccc}
\toprule
\textbf{Method} & \multicolumn{2}{c}{Ctslices} & \multicolumn{2}{c}{Duke} & \multicolumn{2}{c}{w8a} \\
& \# coarse & \# fine & \# coarse & \# fine & \# coarse & \# fine \\
\midrule
\textsc{AML--Newton} (fixed) & 451 & 10 & 48 & 4 & 85 & 8 \\
\textsc{AML--Newton} (random) & 860 & 37 & 45 & 5 & 109 & 6 \\
\textsc{ML--Newton}               & 1,223 & 0 & 373 & 0 & 248 & 0 \\
\bottomrule
\end{tabular}
\end{adjustbox}
\end{table}

\begin{table}[tb]
\centering
\caption{Total number of coarse and fine steps computed by multilevel methods during optimization. The results of this table correspond to \cref{fig:logistic+poisson regression comparison}.}
\label{tab:logistic+poisson regression comparison}
\begin{adjustbox}{max width=\linewidth}
\begin{tabular}{lcccccc}
\toprule
\textbf{Method} & \multicolumn{2}{c}{Gisette} & \multicolumn{2}{c}{Data 311} & \multicolumn{2}{c}{Single-cell RNA} \\
& \# coarse & \# fine & \# coarse & \# fine & \# coarse & \# fine \\
\midrule
\textsc{AML--Newton} (fixed) & 91 & 11 & 1,231 & 64 & 99 & 9 \\
\textsc{AML--Newton} (random) & 94 & 9 & 536 & 54 & 119 & 8 \\
\textsc{ML--Newton}               & 84 & 2 & 953 & 10 & 160 & 0 \\
\bottomrule
\end{tabular}
\end{adjustbox}
\end{table}

\begin{figure}[tb]
    \centering

    % Row 1: CT Slices
    \begin{subfigure}[b]{0.49\textwidth}
    \includegraphics{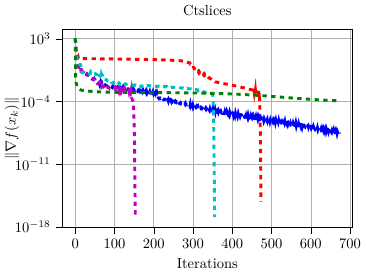}
        \caption{Gradient norm vs. iterations for the  Ctslices dataset with $\ell_1$-regularization}
    \end{subfigure}
    \hfill
    \begin{subfigure}[b]{0.49\textwidth}
      \includegraphics{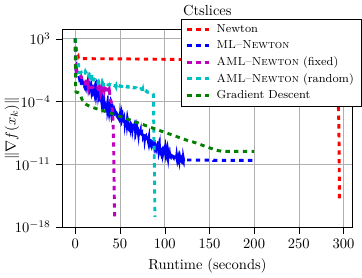}
        \caption{Gradient norm vs. runtime for the Ctslices dataset with $\ell_1$-regularization}
    \end{subfigure}

    % Row 2: Bike
    \begin{subfigure}[b]{0.49\textwidth}
     \includegraphics{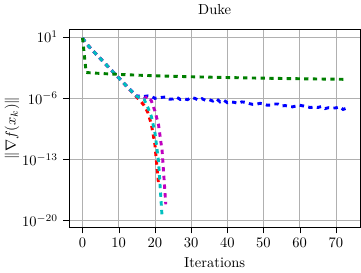}
        \caption{Gradient norm vs. iterations for the Duke dataset}
    \end{subfigure}
    \hfill
    \begin{subfigure}[b]{0.49\textwidth}
      \includegraphics{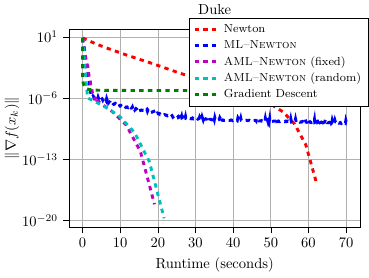}
        \caption{Gradient norm vs. runtime for the Duke dataset}
    \end{subfigure}

    % Row 3: Example
    % \textbf{Example} \par\vspace{0.3em}
    \begin{subfigure}[b]{0.49\textwidth}
        \includegraphics{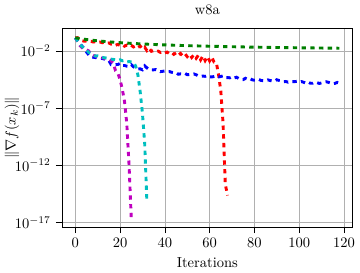}
        \caption{Gradient norm vs. iterations for the w8a dataset with $\mathcal{N}(0,1)$ initialization}
    \end{subfigure}
    \hfill
    \begin{subfigure}[b]{0.49\textwidth}
      \includegraphics{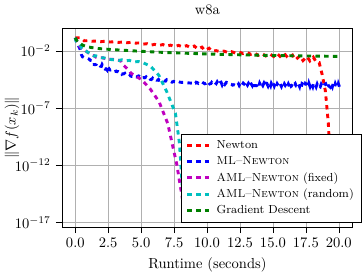}
        \caption{Gradient norm vs. runtime for the w8a dataset with  $\mathcal{N}(0,1)$ initialization}
    \end{subfigure}

    \caption{Comparisons between optimization algorithms on Logistic regression. Each row corresponds to a dataset; the left column shows the gradient norm vs. iterations plots, the right column shows the gradient norm vs. runtime in seconds. See \cref{tab:logistic regression comparison} for the number of fine and coarse steps computed by the multilevel methods.}
    \label{fig:logistic regression comparison}
\end{figure}

\begin{figure}[tbh]
    \centering

    \begin{subfigure}[b]{0.49\textwidth}
        \includegraphics{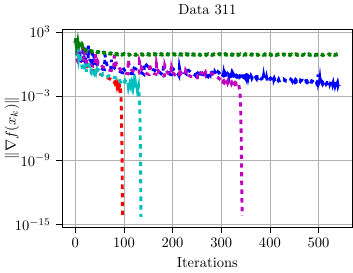}
        \caption{Gradient norm vs. iterations for the Data~311 dataset with Poisson regression and $\mathcal{N}(0,1)$ initialization}
    \end{subfigure}
    \hfill
    \begin{subfigure}[b]{0.49\textwidth}
       \includegraphics{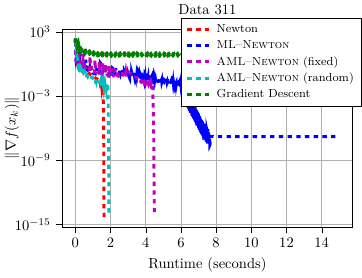}
        \caption{Gradient norm vs. runtime for the Data~311 dataset with Poisson regression and $\mathcal{N}(0,1)$ initialization}
    \end{subfigure}

    \begin{subfigure}[b]{0.49\textwidth}
        \includegraphics{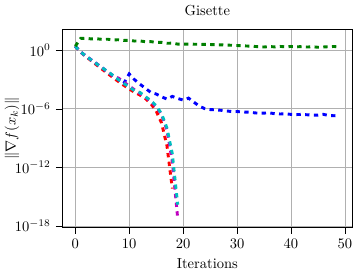}
        \caption{Gradient norm vs. iterations for the Gisette dataset with Logistic regression}
    \end{subfigure}
    \hfill
    \begin{subfigure}[b]{0.49\textwidth}
        \includegraphics{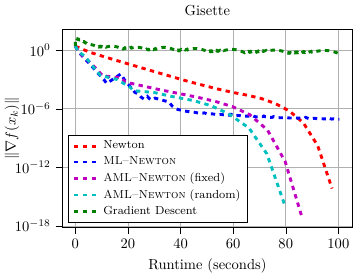}
        \caption{Gradient norm vs. runtime for the Gisette dataset with Logistic regression}
    \end{subfigure}

    % Row 3: Example
    % \textbf{Example} \par\vspace{0.3em}
    \begin{subfigure}[b]{0.49\textwidth}
        \includegraphics{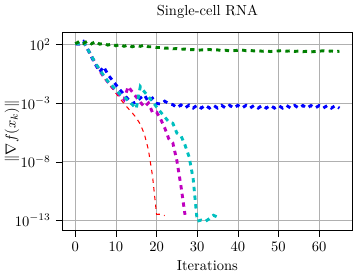}
                \caption{Gradient norm vs. iterations for the Sin\-gle-\-cell RNA dataset with Poisson \mbox{regression}}
    \end{subfigure}
    \hfill
    \begin{subfigure}[b]{0.49\textwidth}
        \includegraphics{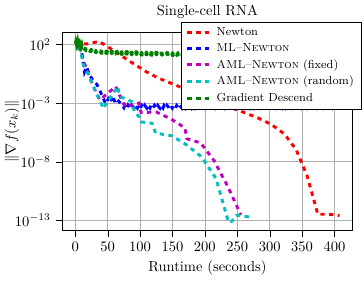}
        \caption{Gradient norm vs. runtime for the Single-\-cell RNA dataset with Poisson regression}
    \end{subfigure}

    \caption{Comparisons between optimization algorithms.  Each row corresponds to a dataset; the left column shows the gradient norm vs. iterations plots, the right column shows the gradient norm vs. runtime in seconds. See \cref{tab:logistic+poisson regression comparison} for the number of fine and coarse steps computed by the multilevel methods.}
    \label{fig:logistic+poisson regression comparison}
\end{figure}

% With these experiments, our main goal is to illustrate that \cref{alg: multilevel} is more efficient than both Newton's and classical multilevel Newton method. In particular, we will see that \cref{alg: multilevel} outperforms Newton's method due its slow first phase. Compared to the classical multilevel Newton method, we demonstrate \cref{alg: multilevel} has smaller total complexity even though it computes trial steps. The results also indicates that it is very likely that methods without local quadratic convergence rates may converge to sub-optimal solutions.

\begin{figure}[tbh]
    \centering

    % Row 1: CT Slices
    \begin{subfigure}[b]{0.49\textwidth}
       \includegraphics{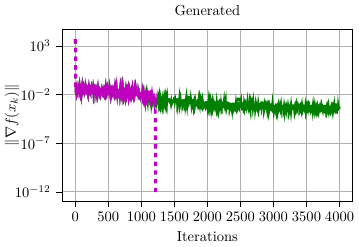}
        \caption{Gradient norm vs. iterations for $\varrho = 0.05$ with $\sigma = 0.01$}
        \label{subfig:small dims iter}
    \end{subfigure}
    \hfill
    \begin{subfigure}[b]{0.49\textwidth}
        \includegraphics{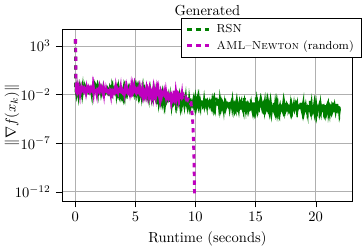}
        \caption{Gradient norm vs. runtime for $\varrho = 0.05$ with $\sigma = 0.01$}
        \label{subfig:small dims sec}
    \end{subfigure}

    % Row 2: Bike
    \begin{subfigure}[b]{0.49\textwidth}
       \includegraphics{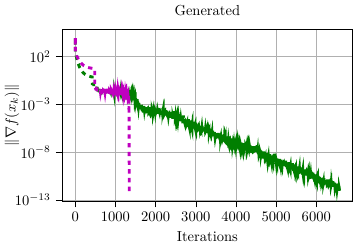}
        \caption{Gradient norm vs. iterations for $\varrho = 0.2$ with $\sigma = 0.1$}
        \label{subfig: subspace fast conv}
    \end{subfigure}
    \hfill
    \begin{subfigure}[b]{0.49\textwidth}
      \includegraphics{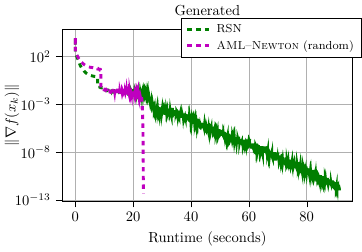}
        \caption{Gradient norm vs. runtime for $\varrho = 0.2$ with $\sigma = 0.1$}
    \end{subfigure}

    % Row 3: Example
    % \textbf{Example} \par\vspace{0.3em}
    \begin{subfigure}[b]{0.49\textwidth}
       \includegraphics{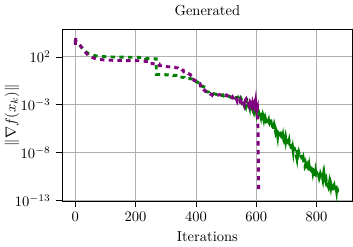}
        \caption{Gradient norm vs. iterations for $\varrho = 0.5$ with $\sigma = 0.1$}
    \end{subfigure}
    \hfill
    \begin{subfigure}[b]{0.49\textwidth}
        \includegraphics{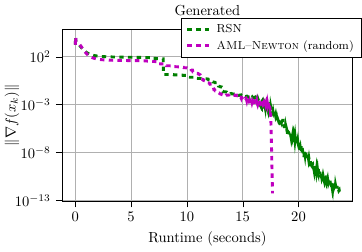}
        \caption{Gradient norm vs. runtime for $\varrho = 0.5$ with $\sigma = 0.1$}
    \end{subfigure}

    \caption{Comparison between \textsc{AML--Newton} (random) with $m=2$ and $n_1 = \varrho n$ and the RSN method with $n_{\rm low} = \varrho n$ ($\varrho \in (0,1]$) on an $\ell_1$-regularized Poisson  regression for the Generated data. Each row corresponds to the size of the coarse model; the left column shows the gradient norm vs. iterations, the right column the gradient norm vs. runtime in seconds. See \cref{tab:generated-steps} for the number of fine and coarse steps computed by
    the \textsc{AML--Newton} method.}
    \label{fig:generated subspace comparisons}
\end{figure}

In \cref{fig:generated_dataset_comparison} we illustrate the efficiency of \textsc{AML--Newton} for problems exhibiting a low-rank structure. By construction, the Hessians of the Poisson model using the Generated dataset are expected to be nearly low-rank with only the first ten eigenvalues being significant, while the remaining eigenvalues should all have a magnitude of $10^{-6}$ or less. This fact alone poses difficulties to optimization algorithms finding a trajectory towards the minimizer. In addition, to make the problem even harder to solve, we use $l_1 > 0$ or a poor initial point (see \cref{tab:combined_experiment}). Observe that in both cases the Newton method has a slow linear rate in the first phase and therefore enters its quadratic phase much later. The classical \textsc{ML--Newton} and Gradient Descent methods enjoy fast convergence in their first phase, respectively, however they slow down significantly as they approach the minimizer, therefore converging to sub-optimal solutions. Hence, neither Newton's method, the \textsc{ML--Newton} method, nor Gradient Descent appear suitable for efficiently solving such problems. On the other hand, the figure suggests that \textsc{AML--Newton} combines characteristics of both Newton's method and the \textsc{ML--Newton} method. Specifically, it exhibits very fast convergence in the initial phase and, due to its provable local quadratic convergence rate, it rapidly approaches the minimizer. Notably, the runtime of both \textsc{A\textsc{ML--Newton}} methods and the \textsc{ML--Newton} method are similar during the initial phase. This indicates that most of the trial steps for \textsc{A\textsc{ML--Newton}} are performed in the quadratic phase. The similar behavior in the early stages is also due to the shared condition \eqref{ineq: extra condition for experiments}. In the end of this section we will revisit this experiment to corroborate the behavior of \textsc{AML–Newton} methods relative to the RSN method.

Moreover, Figures~\ref{fig:logistic regression comparison} and~\ref{fig:logistic+poisson regression comparison} present a comparison between algorithms for logistic and Poisson regression on real-world datasets. These figures demonstrate that many practical problems exhibiting low-rank structures significantly slow down Newton's first phase convergence. In such cases, multilevel and Gradient Descent methods enjoy a much faster rate in their first phase compared to Newton's method. Thus, the \textsc{A\textsc{ML--Newton}} method enters its quadratic phase much earlier than Newton's method, which in turn leads to substantial reductions in runtime. Notably, this behavior is consistently observed for the logistic loss as shown in Figure~\ref{fig:logistic regression comparison}. As anticipated, Gradient Descent struggles to achieve high accuracy in most cases. On the other hand, the experiments in \cref{fig:logistic+poisson regression comparison} suggest that the Hessian matrix incorporates important second-order information that cannot be ignored during optimization. Interestingly, although Newton's first phase is sometimes faster, the results show that \textsc{A\textsc{ML--Newton}} can still reach the solution more quickly in terms of runtime due to its faster first phase. We observe that increasing the difficulty of the problems, for instance, by requiring a sparse solution or using a poor initialization, the results are even more favorable for \textsc{A\textsc{ML--Newton}} when compared to its counterparts.

In the final experiment, we revisit the regularized low-rank problem from \cref{fig:generated_dataset_comparison} to compare \textsc{A\textsc{ML--Newton}} (random) with the RSN method \cite{gower2019rsn, tsipinakis2024multilevel}. Prior work shows that subspace and multilevel methods are well suited to this problem class and can outperform state-of-the-art methods \cite{ho2019newton, tsipinakis2023lowrank}. With this comparison, our aim is to demonstrate that \textsc{A\textsc{ML--Newton}} further improves the performance of second-order methods. Like \textsc{A\textsc{ML--Newton}} (random), RSN will generate the random subspaces according to \cref{def: R_i} at each iteration. For \textsc{A\textsc{ML--Newton}} we set $m=2$ models, that is one coarse and one fine. This configuration makes \textsc{A\textsc{ML--Newton}} closely related to Randomized Subspace Newton, with the key difference that it dynamically chooses whether to compute directions in the subspace or in the full space. The result regarding the convergence behavior of the methods are presented in \cref{fig:generated subspace comparisons}, and \cref{tab:generated-steps} shows the number of coarse and fine steps computed by \textsc{A\textsc{ML--Newton}} during the optimization.

The results on this problem class clearly show that \textsc{A\textsc{ML--Newton}} substantially improves the convergence behavior of the RSN method, both in convergence rate and runtime, primarily due to its local quadratic convergence. In the early iterations the two methods behave similarly; thereafter, \textsc{A\textsc{ML--Newton}} enters its quadratic phase and reaches the minimizer in a few iterations, whereas the RSN method approaches it only linearly. The benefit of adaptively choosing the space is even more pronounced when the subspace dimension is small, a regime central to subspace and multilevel methods. In such cases these methods often progress very slowly near the minimizer and may converge to suboptimal points (see \cref{subfig:small dims iter} and \cref{subfig:small dims sec}). However, as the results indicate, \textsc{A\textsc{ML--Newton}} mitigates this issue. Moreover, \cref{tab:generated-steps} shows that \textsc{A\textsc{ML--Newton}} uses only a small number of fine steps, mostly late in the optimization, so the method effectively behaves like RSN initially and then automatically transitions to full Newton near the solution. 

\begin{table}[t]
\centering
\caption{Number of coarse and fine steps taken by \textsc{A\textsc{ML--Newton}} (random) with \(m=2\) models (one coarse and one fine) on the Generated dataset. Each row corresponds to a different coarse-model size. The results correspond to \cref{fig:generated subspace comparisons}.}
\label{tab:generated-steps}
\begin{tabular}{lcc}
\toprule
\textbf{Method} & \multicolumn{2}{c}{Generated} \\
& \# coarse & \# fine \\
\midrule
\textsc{A\textsc{ML--Newton}}, $n_1 = 0.05n, \sigma = 0.01$ & 1,215 & 19 \\
\textsc{A\textsc{ML--Newton}}, $n_1 = 0.2n, \sigma = 0.1$   & 1,502 & 24 \\
\textsc{A\textsc{ML--Newton}}, $n_1 = 0.5n, \sigma = 0.1$   & 585 & 10 \\
\bottomrule
\end{tabular}
\end{table}

Observe that with $m=2$, whenever \textsc{A\textsc{ML--Newton}} does not take a Newton step, it computes at most one coarse step per iteration (see \cref{tab:generated-steps})). For the case $n_1=0.2n$, the corresponding computational complexities can be estimated as $1{,}502\,N n_1^3 + 24\,Nn^3$ and $6{,}578\,Nn_1^3$ floating point operations for \textsc{A\textsc{ML--Newton}} and RSN, respectively, where $N >0$ is the complexity constant related to the solution of linear systems with the Hessians. Substituting $n_1=0.2n$ yields about
\begin{equation*}
    36\,Nn^3 \quad \text{and} \quad 53\,Nn^3
\end{equation*}
total floating point operations, 
respectively. These figures are even more favorable for \textsc{AML--Newton} when $n_1$ is smaller. Hence, the results of this paper show that a quadratically convergent second-order method can achieve a lower overall computational cost than the Subspace Newton method.

\section{Conclusions}

We developed a new multilevel Newton-type method and established its local quadratic convergence rate for strongly convex and self-concordant functions. The convergence rates are comparable to those of Newton's method and are straightforward to interpret. Experimental results demonstrated that these rates are achievable in practice and that our method requires significantly cheaper iterations than Newton's method. Consequently, the proposed method consistently outperforms Newton's method on problems where the latter struggles due to slow convergence in the early stages of the minimization process. These findings suggest that second-order methods can, in fact, outperform first-order methods even in the initial phases of minimizing strongly convex functions.

Several research questions arise from the results of this paper, which we plan to explore in future work and which we hope will attract interest from the optimization community:
\begin{enumerate}
\item The per-iteration cost depends on coarse-level design. Our simple construction makes \textsc{AML--Newton} practically efficient, matching \textsc{ML--Newton} and RSN in the first phase, while richer coarse models and a more sophisticated rule for selecting  $i^*$ (Step~\ref{step: find} in \cref{alg: multilevel}) are promising directions to further reduce runtime. 

\item For \textsc{RSN} we establish local quadratic convergence with high probability, but we have not yet observed this rate empirically. This may stem from the factor \(\big((1-\varepsilon)/(1+\varepsilon)\big)^5\) in \cref{cor: probabilistic rate strong convexity}, which can delay entry into the quadratic phase, and from standard random projections (Gaussian, 1-hashing, \(s\)-hashing) that may not be well-suited. Nonetheless, we believe that our theory provides a foundation for designing random matrices that realize the proven probabilistic quadratic rates.

\item A natural extension is a randomized-subspace \textsc{AML--Newton} that alternates between two subspaces (instead of subspace--full space): it starts in a very small subspace and automatically switches to a larger one once progress slows, with the larger subspace sized so subspace Newton retains fast linear convergence to high accuracy (\cref{subfig: subspace fast conv}). This is a direct extension of \cref{alg: multilevel} to a full subspace setting for large-scale optimization.

\item Trial steps are central to trust-region methods. We will extend our ideas to Trust-Region Newton to improve convergence in nonconvex problems, and we expect the same principles to enhance other second-order schemes, such as cubic-regularized (Cubic) Newton.

\end{enumerate}

\section*{Code and Data Availability}
Python code and data that was used to produce the numerical results of this paper are available at
\begin{center}
 \url{https://github.com/Ntsip/Adaptive-Multilevel-Newton-method}.
\end{center}

\section*{CRediT Author Statement}
\textbf{Nick Tsipinakis:} Conceptualization, Methodology, Software, Writing - Original Draft. \textbf{Panos Parpas:} Writing - Review \& Editing. \textbf{Matthias Voigt:} Writing - Review \& Editing, Supervision.

%\FloatBarrier

\bibliographystyle{siamplain}
\bibliography{references}
\end{document}